\documentclass[12pt]{amsart}
\usepackage{fancybox}
\usepackage{shuffle}
\usepackage{amsmath}
\usepackage{amsfonts}
\usepackage{amssymb}
\usepackage[all]{xy}           %xypic macro for latex2.09
\usepackage{bm}
\usepackage{bbding}
\usepackage{txfonts}
\usepackage{amscd}

\usepackage{xspace}
\usepackage[shortlabels]{enumitem}
\usepackage{ifpdf}

\ifpdf
  \usepackage[colorlinks,final,backref=page,hyperindex]{hyperref}
\else
  \usepackage[colorlinks,final,backref=page,hyperindex,hypertex]{hyperref}
\fi
\usepackage{tikz}
\usepackage[active]{srcltx}

\makeatletter

\newtheorem{thm}{Theorem}[section]
\newtheorem{lem}[thm]{Lemma}
\newtheorem{cor}[thm]{Corollary}
\newtheorem{pro}[thm]{Proposition}
\newtheorem{ex}[thm]{Example}
\theoremstyle{definition}
\newtheorem{rmk}[thm]{Remark}
\newtheorem{defi}[thm]{Definition}

\newcommand{\nc}{\newcommand}
\newcommand{\delete}[1]{}

\nc{\mlabel}[1]{\label{#1}}  % Use this to suppress names
\nc{\mcite}[1]{\cite{#1}}  % Use this to suppress names
\nc{\mref}[1]{\ref{#1}}  % Use this to suppress names
\nc{\mbibitem}[1]{\bibitem{#1}} % Use this to show number
\delete{% Use the next two lines to show names
\nc{\mlabel}[1]{\label{#1}{\hfill \hspace{1cm}{\bf{{\ }\hfill(#1)}}}}
\nc{\mcite}[1]{\cite{#1}{{\em{{\ }(#1)}}}}  % Use this lines to show names
\nc{\mref}[1]{\ref{#1}{{\em{{\ }(#1)}}}}  % Use this lines to show names
\nc{\mbibitem}[1]{\bibitem[\em #1]{#1}} % Use this to show name
}

\newcommand {\emptycomment}[1]{}

\nc{\oprn}{\theta}
\nc{\Oprn}{\Theta}

\nc{\calo}{\mathcal{O}}
\nc{\oop}{$\mathcal{O}$-operator\xspace}
\nc{\oops}{$\mathcal{O}$-operators\xspace}
\nc{\mrho}{{\bm{\varrho}}}
\nc{\emk}{\mathbf{K}}
\nc{\invlim}{\displaystyle{\lim_{\longleftarrow}}\,}
\nc{\ot}{\otimes}

\newcommand{\be }{\begin{equation}}
\newcommand{\ee }{\end{equation}}

\newcommand{\g}{\mathfrak g}
\newcommand{\h}{\mathfrak h}
\newcommand{\m}{\mathfrak m}

\newcommand{\huaB}{\mathcal{B}}

\newcommand{\huaR}{\mathcal{R}}

\newcommand{\huaG}{\mathcal{G}}

\newcommand{\huaV}{\mathcal{V}}

\newcommand{\huaC}{{\mathcal{C}}}

\newcommand{\huaH}{\mathcal{H}}

\newcommand{\huaO}{{\mathcal{O}}}

\newcommand{\huaZ}{\mathcal{Z}}

\newcommand{\frkI}{\mathfrak I}

\newcommand{\frkL}{\mathfrak L}

\newcommand{\frkR}{\mathfrak R}

\newcommand{\frkT}{\mathfrak T}

\newcommand{\frkX}{\mathfrak X}

\newcommand{\half}{\frac{1}{2}}

\newcommand{\Courant}[1]{\left\llbracket  #1\right\rrbracket }

\newcommand{\Id}{{\rm{Id}}}

\newcommand{\br}[1]{   [ \cdot,    \cdot  ]   }
\newcommand{\ltp}[1]{\Courant{\cdot,\cdot,\cdot}}

\newcommand{\Hom}{\mathrm{Hom}}

\newcommand{\Nij}{\mathsf{Nij}}
\newcommand{\Ob}{\mathsf{Ob}}

\newcommand{\gl}{\mathfrak {gl}}

\newcommand{\Ker}{\mathrm{Ker}}

\newcommand{\ad}{\mathrm{ad}}

\newcommand{\de}{\mathrm{d}}

\nc{\CV}{\mathbf{C}}

\newcommand{\LYA}{Lie-Yamaguti algebra}

\begin{document}

\title{Cohomology and deformations of relative Rota-Baxter operators on Lie-Yamaguti algebras}

\author{Jia Zhao}
\address{Jia Zhao, School of Sciences, Nantong University, Nantong, 226019, Jiangsu, China}
\email{zhaojia@ntu.edu.cn}

\author{Yu Qiao*}
\address{Yu Qiao (corresponding author), School of Mathematics and Statistics, Shaanxi Normal University, Xi'an, 710119, Shaanxi, China}
\email{yqiao@snnu.edu.cn}

\date{\today}

\thanks{{\em Mathematics Subject Classification} (2020): Primary 17B38; Secondary 17B60,17A99}

\thanks{Qiao was partially supported by NSFC grant 11971282.}

\begin{abstract}
 In this paper, we establish the cohomology of relative Rota-Baxter operators on Lie-Yamaguti algebras via the Yamaguti cohomology. Then we use this type of cohomology to characterize deformations of relative Rota-Baxter operators on Lie-Yamaguti algebras. We show that if two linear or formal deformations of a relative Rota-Baxter operator are equivalent, then their infinitesimals are in the same cohomology class in the first cohomology group. Moreover, an order $n$ deformation of a relative Rota-Baxter operator can be extended  to an order $n+1$ deformation if and only if the obstruction class in the second cohomology group is trivial.
\end{abstract}

%\subjclass[2020]{17B38,81R50,17B56,81R12,16T26,17A30,17B62}

\keywords{Lie-Yamaguti algebra, relative Rota-Baxter operator, cohomology, deformation}

\maketitle

%\keywords{{\em Keywords}: Lie-Yamaguti algebra, relative Rota-Baxter operator, cohomology, deformation}

\smallskip
%\vspace{-1.1cm}

%\subjclass{\em Mathematics Subject Classification} (2020): Primary 17B38; Secondary 17B60,17A99

\tableofcontents

\allowdisplaybreaks

 \section{Introduction}
 This paper is devoted to the cohomology theory and deformations of relative Rota-Baxter operators on Lie-Yamaguti algebras.

 Lie-Yamaguti algebras are a generalization of Lie algebras and Lie triple systems, which can be traced back to Nomizu's work on the invariant affine connections on homogeneous spaces in 1950's (\cite{Nomizu}) and Yamaguti's work on general Lie triple systems and Lie triple algebras (\cite{Yamaguti1,Yamaguti2}). Kinyon and Weinstein first called this object a \LYA~  when studying Courant algebroids in the earlier 21st century \cite{Weinstein}. Since then, this system is called a \LYA, which has attracted much attention and is widely investigated recently. For instance, Benito and his collaborators deeply explored irreducible Lie-Yamaguti algebras and their relations with orthogonal Lie algebras \cite{B.B.M,B.D.E,B.E.M1,B.E.M2}. Deformations and extensions of Lie-Yamaguti algebras were examined in \cite{L.CHEN,Ma Y,Zhang1,Zhang2}. Sheng, the first author, and Zhou analyzed product structures and complex structures on Lie-Yamaguti algebras by means of Nijenhuis operators in \cite{Sheng Zhao}. Takahashi studied modules over quandles using representations of Lie-Yamaguti algebras in \cite{Takahashi}.

 Another two topics of the present paper are deformation theory and Rota-Baxter operators, which play important roles in both mathematics and mathematical physics. In mathematics, informally speaking, a deformation of an object is another object that shares the same structure of the original object after a perturbation. Motivated by the foundational work of Kodaira and Spencer \cite{Kodaira} for complex analytic structures, the generalization in algebraic geometry of deformation theory was founded \cite{Hart}. As an application in algebra, Gerstenhaber first studied the deformation theory on associative algebras \cite{Gerstenhaber1,Gerstenhaber2,Gerstenhaber3,Gerstenhaber4,Gerstenhaber5}. Then Nijenhuis and Richardson extended this idea and established the similar results on Lie algebras \cite{Nij1,Nij2}. Deformations of other algebraic structures such as pre-Lie algebras have also been developed \cite{Burde}. In general, deformation theory was set up for binary quadratic operads by Balavoine \cite{bala}.

 Meanwhile, Baxter introduced the notion of Rota-Baxter operators on associative algebras when studying fluctuation theory \cite{Ba}. Then Kupershmidt introduced the notion of $\huaO$-operators (called relative Rota-Baxter operators in the present paper) on Lie algebras  which is a weighted zero Rota-Baxter operator when he found that a relative Rota-Baxter operator is a solution to the classical Yang-Baxter equation in \cite{Kupershmidt}. For more details about classical Yang-Baxter equations and Rota-Baxter operators, one can see \cite{STS} and \cite{Gub}.

Since deformation theory and Rota-Baxter operators have important applications in mathematics and mathematical physics, Sheng and his collaborators developed cohomology theories and deformations of relative Rota-Baxter operators on a series of algebraic structures such as Lie algebras, $3$-Lie algebras and Leibniz algebras. All such cohomology theories obey the following rule: the infinitesimal of a formal deformation is characterized by the cohomology class in the first cohomology group, and the obstruction class is trivial if and only if an order $n$ deformation can be extended to an order $n+1$ deformation. See \cite{TBGS,THS,T.S2} for more details. Recently, Sheng and the first author introduced the notion of relative Rota-Baxter operators on Lie-Yamaguti algebras and revealed the fact that a pre-Lie-Yamaguti algebra is the underlying algebraic structure of relative Rota-Baxter operators \cite{SZ1}.

By the virtue of Lie-Yamaguti algebras and relative Rota-Baxter operators, it is natural to ask the following question: Does there exist an appropriate cohomology theory of relative Rota-Baxter operators on Lie-Yamaguti algebras, which can be used to classify certain types of deformations? We tackle this problem as follows.

The most important step is to construct a suitable cohomology theory for relative Rota-Baxter operators on Lie-Yamaguti algebras. Let $(\g,\br,\Courant{\cdot,\cdot,\cdot})$ denote a Lie-Yamaguti algebra and $(V;\rho,\mu)$ a representation of $\g$. It has been proved that in \cite{SZ1} if $T:V\to \g$ is a relative Rota-Baxter operator on  $\g$ with respect to $(V;\rho,\mu)$, then there is a Lie-Yamaguti algebra structure $([\cdot,\cdot]_T,\Courant{\cdot,\cdot,\cdot}_T)$ on $V$. The key role played in this step is to construct a representation of this \LYA~ $(V,[\cdot,\cdot]_T,\Courant{\cdot,\cdot,\cdot}_T)$ on $\g$ (viewed as the representation space), that is, we shall give the explicit formulas of linear maps $\varrho:V\to \gl(\g)$, $\varpi:\otimes^2V\to \gl(\g)$ and  $D_{\varrho,\varpi}$, which are linked with the representation $(V;\rho,\mu)$ and the relative Rota-Baxter operator $T$, such that the triple  $(\g;\varrho,\varpi)$ becomes a representation of \LYA~ $V$ (see Lemma \ref{represent1} and Theorem \ref{represent}).
Consequently, we obtain the corresponding Yamaguti cohomology of $(V,[\cdot,\cdot]_T,\Courant{\cdot,\cdot,\cdot}_T)$ with coefficients in the representation $(\g;\varrho,\varpi)$.
However, note that the cochain complex of Yamaguti cohomology starts only from $1$-cochians, {\em not} from $0$-cochains.
The main {\em difficulty} is to choose $0$-cochains appropriately and build a proper coboundary map from the set of $0$-cochains to that of $1$-cochains.
{ \em Our strategy} is to define the set of $0$-cochains to be $\wedge^2\g$, then construct the coboundary map explicitly (see Proposition \ref{0cocy}).
%which leads us to giving the $0$-cochain in order to get that of relative Rota-Baxter operators. So the second key role of this step is to
%which leads us to define the set of $0$-cochains to be $\wedge^2\g$ and give the coboundary map explicitly (see Proposition \ref{0cocy}).
In this way, we obtain a cochain complex (associated to $V$) starting from $0$-cochains,
%(whereas a usual cochain complex of a Lie-Yamaguti algebra starts from $1$-cochains by comparison),
which gives rise to the cohomology of the relative Rota-Baxter operator $T$ on Lie-Yamaguti algebras $(\g,[\cdot,\cdot],\Courant{\cdot,\cdot,\cdot})$ (see Definition \ref{cohomology}). A Lie-Yamaguti algebra owns two algebraic operations, which makes its cochain complex much more complicated than others, while other algebras such as Lie algebras, pre-Lie algebras, Leibniz algebras or even $3$-Lie algebras owns only one structure map. As a result, the computation is technical in defining the cohomology of relative Rota-Baxter operators.

The next step is to make use of the cohomology theory to investigate deformations of relative Rota-Baxter operators on Lie-Yamaguti algebras. We consider three kinds of deformations: linear, formal, and higher order deformations. It turns out that our cohomology theory satisfies the rule that is mentioned above and works well (see Theorem \ref{thm1}, Theorem \ref{thm2},  and Theorem \ref{ob}).

 As stated before, a Lie triple system is a spacial case of a Lie-Yamaguti algebra, so the conclusions in the present paper can also be  adapted to the Lie triple system context. See \cite{CHMM} for more details about cohomologies and deformations of relative Rota-Baxter operators on Lie triple systems. However, unlike other algebras such as Lie algebras or Leibniz algebras, there does not exist a suitable graded Lie algebra whose Maurer-Cartan elements are just the Lie-Yamaguri algebra structure by now, thus we do not find a suitable $L_\infty$-algebra that controls the deformations of relative Rota-Baxter operators.  We will overcome this problem in the future and also expect new findings in this direction.

The paper is structured as follows.
In Section 2, we recall some basic concepts including those of Lie-Yamaguti algebras, representations and cohomologies.
In Section 3, the cohomology theory of relative Rota-Baxter operators on Lie-Yamaguti algebras is constructed by using that of Lie-Yamaguti algebras.
Finally, we utilize our established cohomology theory to analyze three kinds of deformations of relative Rota-Baxter operators on Lie-Yamaguti algebras: namely linear, formal, and higher order deformations in Section 4.

 \vspace{2mm}

 \noindent
 {\bf Acknowledgements:} We would like to thank Prof. Yunhe Sheng and Rong Tang of Jilin University for their useful comments and discussions.

\section{Preliminaries: Lie-Yamaguti algebras, representations and cohomologies}
In this paper, all vector spaces are assumed to be over a field $\mathbb{K}$ of characteristic $0$ and finite-dimensional.

In this section, we recall some basic notions such as \LYA s, representations and their cohomology theories.
The notion of  Lie-Yamaguti algebras was first defined by  Yamaguti in \cite{Yamaguti1,Yamaguti2}.

\begin{defi}\cite{Weinstein}\label{LY}
A {\bf \LYA} is a vector space $\g$ equipped with a bilinear bracket $[\cdot,\cdot]:\wedge^2  \mathfrak{g} \to \mathfrak{g} $ and a trilinear bracket $\Courant{\cdot,\cdot,\cdot}:\wedge^2\g \otimes  \mathfrak{g} \to \mathfrak{g} $, which meet the following conditions: for all $x,y,z,w,t \in \g$,
\begin{eqnarray}
~ &&\label{LY1}[[x,y],z]+[[y,z],x]+[[z,x],y]+\Courant{x,y,z}+\Courant{y,z,x}+\Courant{z,x,y}=0,\\
~ &&\Courant{[x,y],z,w}+\Courant{[y,z],x,w}+\Courant{[z,x],y,w}=0,\\
~ &&\label{LY3}\Courant{x,y,[z,w]}=[\Courant{x,y,z},w]+[z,\Courant{x,y,w}],\\
~ &&\Courant{x,y,\Courant{z,w,t}}=\Courant{\Courant{x,y,z},w,t}+\Courant{z,\Courant{x,y,w},t}+\Courant{z,w,\Courant{x,y,t}}.\label{fundamental}
\end{eqnarray}
In the sequel, we denote a \LYA ~by $(\g,\br,\ltp))$
\end{defi}

\begin{ex}
Let $(\g,\br))$ be a Lie algebra. Define a trilinear bracket
$$\ltp::\wedge^2\g\ot \g\to \g$$
by
$$\Courant{x,y,z}:=[[x,y],z],\quad \forall x,y,z \in \g.$$
Then by a direct computation, we know that $(\g,\br,\ltp))$ forms a \LYA.
\end{ex}

The following example is even more interesting.
\begin{ex}
Let $M$ be a closed manifold with an affine connection, and denote by $\frkX(M)$ the set of vector fields on $M$. For all $x, y, z \in \frkX(M) $, set
\begin{eqnarray*}
[x,y]&:=&-T(x,y),\\
\Courant{x,y,z}&:=&-R(x,y)z,
\end{eqnarray*}
where $T$ and $R$ are torsion tensor and curvature tensor respectively. It turns out that the triple
$ (\frkX(M),[\cdot,\cdot],\Courant{\cdot,\cdot,\cdot})$ forms a \LYA. See \cite{Nomizu} for more details.
\end{ex}
\emptycomment{
\begin{rmk}
Given a Lie-Yamaguti algebra $(\m,[\cdot,\cdot]_\m,\Courant{\cdot,\cdot,\cdot}_\m)$ and any two elements $x,y \in \m$, the linear map $D(x,y):\m \to \m,~z\mapsto D(x,y)z=\Courant{x,y,z}_\m$ is an (inner) derivation. Moreover, let $D(\m,\m)$ be the linear span of the inner derivations. Consider the vector space $\g(\m)=D(\m,\m)\oplus \m$, and endow it with a Lie bracket as follows: for all $x,y,z,t \in \m$
\begin{eqnarray*}
[D(x,y),D(z,t)]_{\g(\m)}&=&D(\Courant{x,y,z}_\m,t)+D(z,\Courant{x,y,t}_\m),\\
~[D(x,y),z]_{\g(\m)}&=&D(x,y)z=\Courant{x,y,z}_\m,\\
~[z,t]_{\g(\m)}&=&D(z,t)+[z,t]_\m.
\end{eqnarray*}
Then $(\g(\m),[\cdot,\cdot]_{\g(\m)})$ becomes a Lie algebra.
\end{rmk}}

The following two notions are standard.

\begin{defi}\cite{Sheng Zhao,Takahashi}
Suppose that $(\g,[\cdot,\cdot]_{\g},\Courant{\cdot,\cdot,\cdot}_{\g})$ and $(\h,[\cdot,\cdot]_{\h},\Courant{\cdot,\cdot,\cdot}_{\h})$ are two Lie-Yamaguti algebras. A {\bf homomorphism} from $(\g,[\cdot,\cdot]_{\g},\Courant{\cdot,\cdot,\cdot}_{\g})$ to $(\h,[\cdot,\cdot]_{\h},\Courant{\cdot,\cdot,\cdot}_{\h})$ is a linear map $\phi:\g \to \h$ that preserves the \LYA ~structures, that is, for all $x,y,z \in \g$,
\begin{eqnarray*}
\phi([x,y]_{\g})&=&[\phi(x),\phi(y)]_{\h},\\
~ \phi(\Courant{x,y,z}_{\g})&=&\Courant{\phi(x),\phi(y),\phi(z)}_{\h}.
\end{eqnarray*}
If, moreover, $\phi$ is a bijection, it is then called an {\bf isomorphism}.
\end{defi}

\begin{defi}\cite{Yamaguti2}\label{defi:representation}
Let $(\g,[\cdot,\cdot],\Courant{\cdot,\cdot,\cdot})$ be a Lie-Yamaguti algebra. A {\bf representation} of $\g$ is a vector space $V$ equipped with a linear map $\rho:\g \to \gl(V)$ and a bilinear map $\mu:\otimes^2 \g \to \gl(V)$, which meet the following conditions: for all $x,y,z,w \in \g$,
\begin{eqnarray}
~&&\label{RLYb}\mu([x,y],z)-\mu(x,z)\rho(y)+\mu(y,z)\rho(x)=0,\\
~&&\label{RLYd}\mu(x,[y,z])-\rho(y)\mu(x,z)+\rho(z)\mu(x,y)=0,\\
~&&\label{RLYe}\rho(\Courant{x,y,z})=[D_{\rho,\mu}(x,y),\rho(z)],\\
~&&\label{RYT4}\mu(z,w)\mu(x,y)-\mu(y,w)\mu(x,z)-\mu(x,\Courant{y,z,w})+D_{\rho,\mu}(y,z)\mu(x,w)=0,\\
~&&\label{RLY5}\mu(\Courant{x,y,z},w)+\mu(z,\Courant{x,y,w})=[D_{\rho,\mu}(x,y),\mu(z,w)],
\end{eqnarray}
where the bilinear map $D_{\rho,\mu}:\otimes^2\g \to \gl(V)$ is given by
\begin{eqnarray}
 D_{\rho,\mu}(x,y):=\mu(y,x)-\mu(x,y)+[\rho(x),\rho(y)]-\rho([x,y]), \quad \forall x,y \in \g.\label{rep}
 \end{eqnarray}
It is obvious that $D_{\rho,\mu}$ is skew-symmetric. We denote a representation of $\g$ by $(V;\rho,\mu)$, or we call $(\rho,\mu)$ a representation of $\g$ on $V$ in the sequel.
\end{defi}

\begin{rmk}\label{rmk:rep}
Let $(\g,[\cdot,\cdot],\Courant{\cdot,\cdot,\cdot})$ be a Lie-Yamaguti algebra and $(V;\rho,\mu)$ a representation of $\g$. If $\rho=0$ and the Lie-Yamaguti algebra $\g$ reduces to a Lie tripe system $(\g,\Courant{\cdot,\cdot,\cdot})$,  then the representation reduces to that of the Lie triple system $(\g,\Courant{\cdot,\cdot,\cdot})$: $(V;\mu)$. If $\mu=0$, $D_{\rho,\mu}=0$ and the Lie-Yamaguti algebra $\g$ reduces to a Lie algebra $(\g,[\cdot,\cdot])$, then the representation reduces to that of the Lie algebra $(\g,[\cdot,\cdot])$: $(V;\rho)$. Hence a representation of a Lie-Yamaguti algebra is a natural generalization of that of a Lie algebra or of a Lie triple system.
\end{rmk}

By a direct computation, we have the following lemma.
\begin{lem}
Suppose that $(V;\rho,\mu)$ is a representation of a Lie-Yamaguti algebra $(\g,[\cdot,\cdot],\Courant{\cdot,\cdot,\cdot})$. Then the following equalities are satisfied:
\begin{eqnarray}
\label{RLYc}&&D_{\rho,\mu}([x,y],z)+D_{\rho,\mu}([y,z],x)+D_{\rho,\mu}([z,x],y)=0;\\
\label{RLY5a}&&D_{\rho,\mu}(\Courant{x,y,z},w)+D_{\rho,\mu}(z,\Courant{x,y,w})=[D_{\rho,\mu}(x,y),D_{\rho,\mu}(z,w)];\\
~ &&\mu(\Courant{x,y,z},w)=\mu(x,w)\mu(z,y)-\mu(y,w)\mu(z,x)-\mu(z,w)D_{\rho,\mu}(x,y).\label{RLY6}
\end{eqnarray}
\end{lem}

\begin{ex}\label{ad}
Let $(\g,[\cdot,\cdot],\Courant{\cdot,\cdot,\cdot})$ be a Lie-Yamaguti algebra. We define linear maps $\ad:\g \to \gl(\g)$ and $\frkR :\otimes^2\g \to \gl(\g)$ by $x \mapsto \ad_x$ and $(x,y) \mapsto \mathfrak{R}_{x,y}$ respectively, where $\ad_xz=[x,z]$ and $\mathfrak{R}_{x,y}z=\Courant{z,x,y}$ for all $z \in \g$. Then $(\ad,\mathfrak{R})$ forms a representation of $\g$ on itself, where $\frkL:= D_{\ad,\frkR}$ is given by
\begin{eqnarray*}
\frkL_{x,y}=\mathfrak{R}_{y,x}-\mathfrak{R}_{x,y}+[\ad_x,\ad_y]-\ad_{[x,y]}, \quad \forall x,y \in \g.
\end{eqnarray*}
By \eqref{LY1}, we have
\begin{eqnarray}
\frkL_{x,y}z=\Courant{x,y,z}, \quad \forall z \in \g.\label{lef}
\end{eqnarray}
In this case, $(\g;\ad,\frkR)$ is called the {\bf adjoint representation} of $\g$.
\end{ex}

The representations of Lie-Yamaguti algebras can be characterized by the semidirect Lie-Yamaguti algebras. This fact is revealed via the following proposition.

\begin{pro}\cite{Zhang1}
Let $(\g,[\cdot,\cdot],\Courant{\cdot,\cdot,\cdot})$ be a Lie-Yamaguti algebra and $V$ a vector space. Let $\rho:\g \to \gl(V)$ and $\mu:\otimes^2 \g \to \gl(V)$ be linear maps. Then $(V;\rho,\mu)$ is a representation of $(\g,[\cdot,\cdot],\Courant{\cdot,\cdot,\cdot})$ if and only if there is a Lie-Yamaguti algebra structure $([\cdot,\cdot]_{\rho,\mu},\Courant{\cdot,\cdot,\cdot}_{\rho,\mu})$ on the direct sum $\g \oplus V$ which is defined by for all $x,y,z \in \g, ~u,v,w \in V$,
\begin{eqnarray}
\label{semi1}[x+u,y+v]_{\rho,\mu}&=&[x,y]+\rho(x)v-\rho(y)u,\\
\label{semi2}~\Courant{x+u,y+v,z+w}_{\rho,\mu}&=&\Courant{x,y,z}+D_{\rho,\mu}(x,y)w+\mu(y,z)u-\mu(x,z)v,
\end{eqnarray}
where $D_{\rho,\mu}$ is given by \eqref{rep}.
This Lie-Yamaguti algebra $(\g \oplus V,[\cdot,\cdot]_{\rho,\mu},\Courant{\cdot,\cdot,\cdot}_{\rho,\mu})$ is called the {\bf semidirect product Lie-Yamaguti algebra}, and is denoted by $\g \ltimes_{\rho,\mu} V$.
\end{pro}
\begin{proof}
The proof is a direct computation, so we omit the details. Or one can see \cite{Zhang1} for more details.
\end{proof}

Let us recall the cohomology theory on Lie-Yamaguti algebras given in \cite{Yamaguti2}. Let $(\g,[\cdot,\cdot],\Courant{\cdot,\cdot,\cdot})$ be a  Lie-Yamaguti algebra and $(V;\rho,\mu)$ a representation of $\g$. We denote the set of $p$-cochains by $C^p_{\rm LieY}(\g,V)~(p \geqslant 1)$, where
\begin{eqnarray*}
C^{n+1}_{\rm LieY}(\g,V)\triangleq
\begin{cases}
\Hom(\underbrace{\wedge^2\g\otimes \cdots \otimes \wedge^2\g}_n,V)\times \Hom(\underbrace{\wedge^2\g\otimes\cdots\otimes\wedge^2\g}_{n}\otimes\g,V), & \forall n\geqslant 1,\\
\Hom(\g,V), &n=0.
\end{cases}
\end{eqnarray*}

In the sequel, we recall the coboundary map of $p$-cochains:
\begin{itemize}
\item If $n\geqslant 1$, for any $(f,g)\in C^{n+1}_{\rm LieY}(\g,V)$, the coboundary map
$$\delta=(\delta_{\rm I},\delta_{\rm II}):C^{n+1}_{\rm LieY}(\g,V)\to C^{n+2}_{\rm LieY}(\g,V),$$
$$\qquad \qquad\qquad \qquad\qquad \quad (f,g)\mapsto(\delta_{\rm I}(f,g),\delta_{\rm II}(f,g)),$$
 is given as follows:
\begin{eqnarray*}
~ &&\Big(\delta_{\rm I}(f,g)\Big)(\frkX_1,\cdots,\frkX_{n+1})\\
~ &=&(-1)^n\Big(\rho(x_{n+1})g(\frkX_1,\cdots,\frkX_n,y_{n+1})-\rho(y_{n+1})g(\frkX_1,\cdots,\frkX_n,x_{n+1})\\
~ &&\quad\quad-g(\frkX_1,\cdots,\frkX_n,[x_{n+1},y_{n+1}])\Big)\\
~ &&+\sum_{k=1}^{n}(-1)^{k+1}D_{\rho,\mu}(\frkX_k)f(\frkX_1,\cdots,\hat{\frkX_k},\cdots,\frkX_{n+1})\\
~ &&+\sum_{1\leqslant k<l\leqslant n+1}(-1)^{k}f(\frkX_1,\cdots,\hat{\frkX_k},\cdots,\frkX_k\circ\frkX_l,\cdots,\frkX_{n+1}),\\
~ &&\\
~ &&\Big(\delta_{\rm II}(f,g)\Big)(\frkX_1,\cdots,\frkX_{n+1},z)\\
~ &=&(-1)^n\Big(\mu(y_{n+1},z)g(\frkX_1,\cdots,\frkX_n,x_{n+1})-\mu(x_{n+1},z)g(\frkX_1,\cdots,\frkX_n,y_{n+1})\Big)\\
~ &&+\sum_{k=1}^{n+1}(-1)^{k+1}D_{\rho,\mu}(\frkX_k)g(\frkX_1,\cdots,\hat{\frkX_k},\cdots,\frkX_{n+1},z)\\
~ &&+\sum_{1\leqslant k<l\leqslant n+1}(-1)^kg(\frkX_1,\cdots,\hat{\frkX_k},\cdots,\frkX_k\circ\frkX_l,\cdots,\frkX_{n+1},z)\\
~ &&+\sum_{k=1}^{n+1}(-1)^kg(\frkX_1,\cdots,\hat{\frkX_k},\cdots,\frkX_{n+1},\Courant{x_k,y_k,z}),
\end{eqnarray*}
where $\frkX_i=x_i\wedge y_i\in\wedge^2\g~(i=1,\cdots,n+1),~z\in \g$ and $\frkX_k\circ\frkX_l:=\Courant{x_k,y_k,x_l}\wedge y_l+x_l\wedge\Courant{x_k,y_k,y_l}$.

\item if $n=0$, for any $f \in C^1_{\rm LieY}(\g,V)$, the coboundary map
$$\delta:C^1_{\rm LieY}(\g,V)\to C^2_{\rm LieY}(\g,V),$$
$$\qquad \qquad \qquad f\mapsto (\delta_{\rm I}(f),\delta_{\rm II}(f)),$$
is defined by
\begin{eqnarray*}
\label{1cochain}\Big(\delta_{\rm I}(f)\Big)(x,y)&=&\rho(x)f(y)-\rho(y)f(x)-f([x,y]),\\
~ \label{2cochain}\Big(\delta_{\rm II}(f)\Big)(x,y,z)&=&D_{\rho,\mu}(x,y)f(z)+\mu(y,z)f(x)-\mu(x,z)f(y)-f(\Courant{x,y,z}),\quad \forall x,y, z\in \g.
\end{eqnarray*}
\end{itemize}

Yamaguti showed the following fact.

\begin{pro}{\rm \cite{Yamaguti2}}
 With the notations above, for any $f\in C^1_{\rm LieY}(\g,V)$, we have
 \begin{eqnarray*}
 \delta_{\rm I}\Big(\delta_{\rm I}(f)),\delta_{\rm II}(f)\Big)=0\quad {\rm and} \quad\delta_{\rm II}\Big(\delta_{\rm I}(f)),\delta_{\rm II}(f)\Big)=0.
 \end{eqnarray*}
 Moreover, for all $(f,g)\in C^p_{\rm LieY}(\g,V),~(p\geqslant 2)$, we have
  \begin{eqnarray*}
  \delta_{\rm I}\Big(\delta_{\rm I}(f,g)),\delta_{\rm II}(f,g)\Big)=0\quad{\rm and} \quad \delta_{\rm II}\Big(\delta_{\rm I}(f,g)),\delta_{\rm II}(f,g)\Big) =0.
  \end{eqnarray*}
  Thus the cochain complex $(C^\bullet_{\rm LieY}(\g,V)=\bigoplus\limits_{p=1}^\infty C^p_{\rm LieY}(\g,V),\delta)$ is well defined. For convenience, we call it the Yamaguti cohomology in this paper.
  \end{pro}

\begin{defi}
With the above notations, let $(f,g)$ in $C^p_{\rm LieY}(\g,V))$ (resp. $f\in C^1_{\rm LieY}(\g,V)$ for $p=1$) be a $p$-cochain. If it satisfies $\delta(f,g)=0$ (resp. $\delta(f)=0$), then it is called a $p$-cocycle. If there exists $(h,s)\in C^{p-1}_{\rm LieY}(\g,V)$,~(resp. $t\in C^1(\g,V)$, if $p=2$) such that $(f,g)=\delta(h,s)$~(resp. $(f,g)=\delta(t)$), then it is called a $p$-coboundary ($p\geqslant 2$). The set of $p$-cocycles and that of $p$-coboundaries is denoted by $Z^p_{\rm LieY}(\g,V)$ and $B^p_{\rm LieY}(\g,V)$ respectively. The resulting $p$-cohomology group is defined to be the factor space
$$H^p_{\rm LieY}(\g,V)=Z^p_{\rm LieY}(\g,V)/B^p_{\rm LieY}(\g,V).$$ In particular, we have
$$H^1_{\rm LieY}(\g,V)=\{f\in C^1_{\rm LieY}(\g,V):\delta (f)=0\}.$$
\end{defi}

\section{Cohomologies of relative Rota-Baxter operators on Lie-Yamaguti algebras}

In this section, we build the cohomology of relative Rota-Baxter operators on Lie-Yamaguti algebras. Once a relative Rota-Baxter operator on a Lie-Yamaguti algebra with respect to a representation is given, we obtain a Lie-Yamaguti algebra structure on the representation space. Then we will construct a representation of the representation space (viewed as a Lie-Yamaguti algebra) on the Lie-Yamaguti algebra as a vector space. At the beginning, we recall some notions and conclusions in \cite{SZ1} about relationship between relative Rota-Baxter operators on Lie-Yamaguti algebras and pre-Lie-Yamaguti algebras.

\begin{defi}{\rm\cite{SZ1}}
Let $(\g,[\cdot,\cdot],\Courant{\cdot,\cdot,\cdot})$ be a Lie-Yamaguti algebra and $(V;\rho,\mu)$ a representation of $\g$. A linear map $T:V\to \g$ is called a {\bf relative Rota-Baxter operator} on $\g$ with respect to $(V;\rho,\mu)$ if $T$ satisfies
\begin{eqnarray}
~\label{Ooperator1}[Tu,Tv]&=&T\Big(\rho(Tu)v-\rho(Tv)u\Big),\\
~\label{Ooperator2}\Courant{Tu,Tv,Tw}&=&T\Big(D_{\rho,\mu}(Tu,Tv)w+\mu(Tv,Tw)u-\mu(Tu,Tw)v\Big), \quad \forall u,v,w \in V.
\end{eqnarray}
\end{defi}

Recall also that in \cite{SZ1} once a pre-Lie-Yamaguti algebra $(A,*,\{\cdot,\cdot,\cdot\})$ is given, there exists a Lie-Yamaguti algebra structure $([\cdot,\cdot]_C,\Courant{\cdot,\cdot,\cdot}_C)$ on $A$ as follows:
\begin{eqnarray*}
[x,y]_C&=&x*y-y*x,\\
\Courant{x,y,z}_C&=&\{x,y,z\}_D+\{x,y,z\}-\{y,x,z\},\quad \forall x,y,z \in A.
\end{eqnarray*}
We call the Lie-Yamaguti algebra $(A,[\cdot,\cdot]_C,\Courant{\cdot,\cdot,\cdot}_C)$ the {\bf sub-adjacent Lie-Yamaguti algebra}.

\begin{pro}{\cite{SZ1}}
Let $T$ be a relative Rota-Baxter operator on a Lie-Yamaguti algebra $(\g,[\cdot,\cdot],\Courant{\cdot,\cdot,\cdot})$ with respect to $(V;\rho,\mu)$. Define
\begin{eqnarray*}
[u,v]_T&=&\rho(Tu)v-\rho(Tv)u,\\
\Courant{u,v,w}_T&=&D_{\rho,\mu}(Tu,Tv)w+\mu(Tv,Tw)u-\mu(Tu,Tw)v,\quad \forall u,v,w\in V.
\end{eqnarray*}
Then $(V,[\cdot,\cdot]_T,\Courant{\cdot,\cdot,\cdot}_T)$ is a Lie-Yamaguti algebra, which is sub-adjacent to the pre-Lie-Yamaguti algebra $(V,*,\{\cdot,\cdot,\cdot\})$ defined by
\begin{eqnarray}
u*v=\rho(Tu)v,\quad \{u,v,w\}=\mu(Tv,Tw)u,\quad \forall u,v, w \in V.\label{preO}
\end{eqnarray}
 Thus $T$ is a homomorphism from $(V,[\cdot,\cdot]_T,\Courant{\cdot,\cdot,\cdot}_T)$ to $(\g,[\cdot,\cdot],\Courant{\cdot,\cdot,\cdot})$.
\end{pro}

So far, the representation space $V$ has been endowed with a Lie-Yamaguti algebra structure $([\cdot,\cdot]_T,\Courant{\cdot,\cdot,\cdot}_T)$. We will give a representation of $V$ on $\g$ (viewed as a vectors space) in the sequel. Define two linear maps $\varrho:V\to \gl(\g)$ and $\varpi:\otimes^2V\to \gl(\g)$ by
\begin{eqnarray}
\label{repre1}\varrho(u)x&:=&[Tu,x]+T\big(\rho(x)u\big),\\
\label{repre2}\varpi(u,v)x&:=&\Courant{x,Tu,Tv}-T\big(D_{\rho,\mu}(x,Tu)v-\mu(x,Tv)u\big), \quad \forall x\in \g,~u,v \in V.
\end{eqnarray}

The following proposition gives the explicit formula of $D_{\varrho,\varpi}$.
\begin{lem}\label{represent1}
Let $T$ be a relative Rota-Baxter operator on a Lie-Yamaguti algebra $(\g,[\cdot,\cdot],\Courant{\cdot,\cdot,\cdot})$ with respect to $(V;\rho,\mu)$. Then with the above notations, we have
\begin{eqnarray}
\label{repre3}D_{\varrho,\varpi}(u,v)x=\Courant{Tu,Tv,x}-T\Big(\mu(Tv,x)u-\mu(Tu,x)v\Big), \quad \forall u,v \in V, ~x\in \g.
\end{eqnarray}
\end{lem}
\begin{proof}
Since $T$ is a relative Rota-Baxter operator, by a direct computation, we have
\begin{eqnarray*}
~ &&D_{\varrho,\varpi}(u,v)x\\
&\stackrel{\eqref{rep}}{=}&\varpi(v,u)x-\varpi(u,v)x+[\varrho(u),\varrho(v)]x-\varrho([u,v]_T)x\\
~ &\stackrel{\eqref{repre1},\eqref{repre2}}{=}&\Courant{x,Tv,Tu}-T\Big(D_{\rho,\mu}(x,Tv)u-\mu(x,Tu)v\Big)-\Courant{x,Tu,Tv}+T\Big(D_{\rho,\mu}(x,Tu)v-\mu(x,Tv)u\Big)\\
~ &&+[Tu,[Tv,x]]+[Tu,T(\rho(x)v)]+T(\rho([Tv,x])u)+T(\rho(T(\rho(x)v)u))\\
~ &&-[Tv,[Tu,x]]-[Tv,T(\rho(x)u)]-T(\rho([Tu,x])v)-T(\rho(T(\rho(x)u)v))\\
~ &&-[T(\rho(Tu)v-\rho(Tv)u),x]-T(\rho(x)\rho(Tu)v)+T(\rho(x)\rho(Tv)u)\\
~ &\stackrel{\eqref{Ooperator1}}{=}&\Courant{x,Tv,Tu}-\Courant{x,Tu,Tv}+[Tu,[Tv,x]]-[Tv,[Tu,x]]-[[Tu,Tv],x]\\
~ &&-T\Big(D_{\rho,\mu}(x,Tv)u-\mu(x,Tu)v\Big)+T\Big(D_{\rho,\mu}(x,Tu)v-\mu(x,Tv)u\Big)\\
~ &&+T(\rho(Tu)\rho(x)v-\rho(T(\rho(x)v)u)-T(\rho(Tv)\rho(x)u-\rho(T(\rho(x)u)v)\\
~ &&+T(\rho([Tv,x])u)+T(\rho(T(\rho(x)v)u))-T(\rho([Tu,x])v)-T(\rho(T(\rho(x)u)v))\\
~ &&-T(\rho(x)\rho(Tu)v)+T(\rho(x)\rho(Tv)u)\\
~ &\stackrel{\eqref{LY1},\eqref{rep}}{=}&\Courant{Tu,Tv,x}-T\Big(\mu(Tv,x)u-\mu(Tu,x)v\Big).
\end{eqnarray*}
The conclusion thus follows.
\end{proof}

\begin{thm}\label{represent}
Let $T$ be a relative Rota-Baxter operator on a Lie-Yamaguti algebra $(\g,[\cdot,\cdot],\Courant{\cdot,\cdot,\cdot})$ with respect to $(V;\rho,\mu)$.
Then $(\g;\varrho,\varpi)$ is a representation of the sub-adjacent Lie-Yamaguti algebra $(V,[\cdot,\cdot]_T,\Courant{\cdot,\cdot,\cdot}_T)$, where $\varrho,~\varpi$ and $D_{\varrho,\varpi}$ are given by \eqref{repre1}-\eqref{repre3} respectively.
\end{thm}

We may prove Theorem \ref{represent} by checking that $(\g;\varrho,\varpi)$ satisfies the conditions in Definition \ref{defi:representation}, but here we will prove it by another way. In order to do this, we should go back some notions in \cite{Sheng Zhao}. Recall that a Nijenhuis operator on a Lie-Yamaguti algebra $(\g,[\cdot,\cdot],\Courant{\cdot,\cdot,\cdot})$ is a linear map $N:\g\to \g$ satisfying
\begin{eqnarray*}
\label{NIJ1}[Nx,Ny]&=&N\big([Nx,y]+[x,Ny]-N[x,y]\big),\\
~\label{NIJ2}\Courant{Nx,Ny,Nz}&=&N\Big(\Courant{Nx,Ny,z}+\Courant{Nx,y,Nz}+\Courant{x,Ny,Nz}\\
~\nonumber &&-N\Courant{Nx,y,z}-N\Courant{x,Ny,z}-N\Courant{x,y,Nz}+N^2\Courant{x,y,z}\Big), \quad \forall x,y,z \in \g.
\end{eqnarray*}
We denote a pair of deformed brackets $([\cdot,\cdot]_N,\Courant{\cdot,\cdot,\cdot}_N)$ by
\begin{eqnarray}
\label{deform1}[x,y]_N&=&[Nx,y]+[x,Ny]-N[x,y],\\
\label{deform2}~\Courant{x,y,z}_N&=&\Courant{Nx,Ny,z}+\Courant{Nx,y,Nz}+\Courant{x,Ny,Nz}\\
\nonumber~ &&-N\Courant{Nx,y,z}-N\Courant{x,Ny,z}-N\Courant{x,y,Nz}+N^2\Courant{x,y,z}, \quad \forall x,y,z \in \g.
\end{eqnarray}
Then $(\g,[\cdot,\cdot]_N,\Courant{\cdot,\cdot,\cdot}_N)$ is a Lie-Yamaguti algebra and $N$ is a Lie-Yamaguti homomorphism from $(\g,[\cdot,\cdot]_N,\Courant{\cdot,\cdot,\cdot}_N)$ to $(\g,[\cdot,\cdot],\Courant{\cdot,\cdot,\cdot})$.

\vspace{2mm}

\noindent
\emph{Proof of Theorem {\rm \ref{represent}:}}
It is direct to see that if $T:V \to \g$ is a relative Rota-Baxter operator on a Lie-Yamaguti algebra $(\g,[\cdot,\cdot],\Courant{\cdot,\cdot,\cdot})$ with respect to a representation $(V;\rho,\mu)$, then
$N_T=\begin{pmatrix}
0 & T\\
0 & 0
\end{pmatrix}$
is a Nijenhuis operator on the semidirect product Lie-Yamaguti algebra $\g\ltimes_{\rho,\mu} V$. Then by \eqref{deform1} and \eqref{deform2}, we deduce that there is a Lie-Yamaguti algebra structure on $V\oplus\g\cong\g \oplus V$ given by for all $x,y,z \in \g,~u,v,w\in V,$
\begin{eqnarray*}
~ &&[x+u,y+v]_{N_T}\\
~ &=&[N_T(x+u),y+v]_{\rho,\mu}+[x+u,N_T(y+v)]_{\rho,\mu}-N_T[x+u,y+v]_{\rho,\mu}\\
~ &=&[Tu,y+v]_{\rho,\mu}+[x+u,Tv]_{\rho,\mu}-N_T([x,y]+\rho(x)v-\rho(y)u)\\
~ &=&[Tu,y]+\rho(Tu)v+[x,Tv]-\rho(Tv)u-T(\rho(x)v-\rho(y)u)\\
~ &=&[u,v]_T+\varrho(u)y-\varrho(v)x,\\
~ &&\\
~&&\Courant{x+u,y+v,z+w}_{N_T}\\
~ &=&\Courant{N_T(x+u),N_T(y+v),z+w}_{\rho,\mu}+\Courant{N_{T}(x+u),y+v,N_{T}(z+w)}_{\rho,\mu}+\Courant{x+u,N_{T}(y+v),N_{T}(z+w)}_{\rho,\mu}\\
~ &&-N_T(\Courant{N_{T}(x+u),y+v,z+w}_{\rho,\mu}+\Courant{x+u,N_{T}(y+v),z+w}_{\rho,\mu})+\Courant{x+u,y+v,N_{T}(z+w)}_{\rho,\mu})\\
~ &=&\Courant{Tu,Tv,z+w}_{\rho,\mu}+\Courant{Tu,y+v,Tw}_{\rho,\mu}+\Courant{x+u,Tv,Tw}_{\rho,\mu}\\
~ &&-N_T(\Courant{Tu,y+v,z+w}_{\rho,\mu}+\Courant{x+u,Tv,z+w}_{\rho,\mu}+\Courant{x+u,y+v,Tw}_{\rho,\mu})\\
~ &=&\Courant{Tu,Tv,z}+D_{\rho,\mu}(Tu,Tv)w+\Courant{Tu,y,Tw}-\mu(Tu,Tw)v+\Courant{x,Tv,Tw}+\mu(Tv,Tw)u\\
~ &&-T\Big(D_{\rho,\mu}(Tu,y)w-\mu(Tu,z)v+D_{\rho,\mu}(x,Tv)w+\mu(Tv,z)u+\mu(y,Tw)u-\mu(x,Tw)v\Big)\\
~ &=&\Courant{u,v,w}_T+D_{\varrho,\varpi}(u,v)z+\varpi(v,w)x-\varpi(u,w)y,
\end{eqnarray*}
which implies that $(\g;\varrho,\varpi)$ is a representation of Lie-Yamaguti algebra $(V,[\cdot,\cdot]_T,\Courant{\cdot,\cdot,\cdot}_T)$. This finishes the proof.\qed

Having endowed the vector space $V$ with a Lie-Yamaguti algebra structure $([\cdot,\cdot]_T,\Courant{\cdot,\cdot,\cdot}_T)$ and established a representation $(\g;\varrho,\varpi)$ of $(V,[\cdot,\cdot]_T,\Courant{\cdot,\cdot,\cdot}_T)$ gives rise to the corresponding Yamaguti cohomology of $(V,[\cdot,\cdot]_T,\Courant{\cdot,\cdot,\cdot}_T)$ with coefficients in the representation $(\g;\varrho,\varpi)$:
$$\Big(\oplus_{p=1}^\infty C_{\rm LieY}^p(V,\g),(\delta_{\rm I}^T,\delta_{\rm II}^T)\Big).$$
 More precisely, if $n\geqslant 1$, $\delta^T:C_{\rm LieY}^{n+1}(V,\g)\to C_{\rm LieY}^{n+2}(V,\g)$ is given by for any $(f,g)\in C_{\rm LieY}^{n+1}(V,\g)$,
\begin{eqnarray*}
~ &&\Big(\delta^T_{\rm I}(f,g)\Big)(\huaV_1,\cdots,\huaV_{n+1})\\
~ &=&(-1)^{n}\Big([Tu_{n+1},g(\huaV_1,\cdots,\huaV_n,v_{n+1})]-[Tv_{n+1},g(\huaV_1,\cdots,\huaV_n,u_{n+1})]\\
~ &&\quad\quad-g(\huaV_1,\cdots,\huaV_n,\rho(Tu_{n+1})v_{n+1}-\rho(Tv_{n+1})u_{n+1})+T\big(\rho(g(\huaV_1,\cdots,\huaV_n,v_{n+1}))u_{n
+1}\big)\\
~ &&\quad\quad-T\big(\rho(g(\huaV_1,\cdots,\huaV_n,u_{n+1}))v_{n
+1}\big)\Big)\\
~ &&+\sum_{k=1}^{n+1}(-1)^{k+1}\Big(\Courant{Tu_{k},Tv_k,f(\huaV_1,\cdots,\hat{\huaV_k},\cdots,\huaV_{n+1})}+T\big(\mu(Tv_k,f(\huaV_1,\cdots,\hat{\huaV_k},\cdots,\huaV_{n+1}))u_k\big)\\
~ &&\quad\quad -T\big(\mu(Tu_k,f(\huaV_1,\cdots,\hat{\huaV_k},\cdots,\huaV_{n+1}))v_k\big)\Big)\\
~ &&+\sum_{1\leqslant k<l\leqslant n+1}(-1)^kf(\huaV_1,\cdots,\hat{\huaV_k},\cdots,\huaV_k\circ\huaV_l,\cdots,\huaV_{n+1}),\\
~ &&\\
~&&\Big(\delta^T_{\rm II}(f,g)\Big)(\huaV_1,\cdots,\huaV_{n+1},w)\\
~ &=&(-1)^n\Big(\Courant{g(\huaV_1,\cdots,\huaV_n,u_{n+1}),Tv_{n+1},Tw}-\Courant{g(\huaV_1,\cdots,\huaV_n,v_{n+1}),Tu_{n+1},Tw}\\
~ &&\quad\quad+T\big(D_{\rho,\mu}(g(\huaV_1,\cdots,\huaV_n,u_{n+1}),Tv_{n+1})w-\mu(g(\huaV_1,\cdots,\huaV_n,u_{n+1}),Tw)v_{n+1}\\
~ &&\quad\quad+D_{\rho,\mu}(g(\huaV_1,\cdots,\huaV_n,v_{n+1}),Tu_{n+1})w-\mu(g(\huaV_1,\cdots,\huaV_n,v_{n+1}),Tw)u_{n+1}\big)\Big)\\
~ &&+\sum_{k=1}^{n+1}(-1)^{k}\Big(\Courant{Tu_k,Tv_k,g(\huaV_1,\cdots,\hat{\huaV_k},\cdots,\huaV_{n+1},w)}+T\big(\mu(Tv_k,g(\huaV_1,\cdots,\hat{\huaV_k},\cdots,\huaV_{n+1},w))u_k\\
~ &&\quad\quad-\mu(Tu_k,g(\huaV_1,\cdots,\hat{\huaV_k},\cdots,\huaV_{n+1},w))v_k\big)\Big)\\
~ &&+\sum_{1\leqslant k<l\leqslant n+1}(-1)^kg(\huaV_1,\cdots,\hat{\huaV_k},\cdots,\huaV_k\circ\huaV_l,\cdots,\huaV_{n+1},w)\\
~ &&+\sum_{k=1}^{n+1}(-1)^kg(\huaV_1,\cdots,\hat{\huaV_k},\cdots,\huaV_{n+1},\Courant{u_k,v_k,w}_T),
\end{eqnarray*}
where $\huaV_i=u_i\wedge v_i\in \wedge^2V~ (1\leqslant i\leqslant n+1),~ w\in V$ and $\huaV_k\circ\huaV_l=\Courant{u_k,v_k,u_l}_T\wedge v_l+u_l\wedge\Courant{u_k,v_k,v_l}_T$.

In particular, for any $f\in C_{\rm LieY}^1(V,\g)=\Hom(V,\g)$,
$$\delta^T:C_{\rm LieY}^1(V,\g)\to C_{\rm LieY}^2(V,\g),\quad f \mapsto (\delta^T_{\rm I}(f),\delta^T_{\rm II}(f))$$
is given by
\begin{eqnarray*}
\Big(\delta^T_{\rm I}(f)\Big)(u,v)&=&[Tu,f(v)]-[Tv,f(u)]+T\Big(\rho(f(v)u)-\rho(f(u)v)\Big)-f([u,v]_T),\\
\Big(\delta^T_{\rm II}(f)\Big)(u,v,w)&=&\Courant{Tu,Tv,f(w)}+\Courant{f(u),Tv,Tw}-\Courant{f(v),Tu,Tw}-f(\Courant{u,v,w}_T)\\
~ &&-T\Big(D_{\rho,\mu}(f(u),Tv)w-D_{\rho,\mu}(f(v),Tu)w+\mu(Tv,f(w))u-\mu(Tu,f(w))v\\
~ &&\quad-\mu(f(u),Tw)v+\mu(f(v),Tw)u\Big), \qquad \forall u,v,w \in V.
\end{eqnarray*}

In the following, we will give the set of $0$-cochains and the explicit coboundary map. For all $\frkX\in \wedge^2\g$, define $\delta(\frkX):V\to \g$ by
\begin{eqnarray}
\label{delta}\delta(\frkX)v:=TD_{\rho,\mu}(\frkX)v-\Courant{\frkX,Tv}, \quad \forall v\in V.
\end{eqnarray}

\begin{pro}\label{0cocy}
Let $T$ be a relative Rota-Baxter operator on a Lie-Yamaguti algebra $(\g,\br,\Courant{\cdot,\cdot,\cdot})$ with respect to a representation $(V;\rho,\mu)$. Then $\delta(\frkX)$ is a $1$-cocycle on the Lie-Yamaguti algebra $(V,[\cdot,\cdot]_T,\Courant{\cdot,\cdot,\cdot}_T)$ with the coefficients in the representation $(\g;\varrho,\varpi)$.
\end{pro}
\begin{proof}
It is sufficient to show that both $\delta_{\rm I}^T(\delta(\frkX))$ and $\delta_{\rm II}^T(\delta(\frkX))$ all vanish. Indeed,  for any $u,v,w\in V$, we have
{\footnotesize
\begin{eqnarray*}
~ &&\delta^T_{\rm I}\Big(\delta(\frkX)\Big)(u,v)\\
~ &\stackrel{\eqref{1cochain}}{=}&\varrho(u)\delta(\frkX)(v)-\varrho(v)\delta(\frkX)(u)-\delta(\frkX)([u,v]_T)\\
~ &\stackrel{\eqref{repre1}}{=}&[Tu,\delta(\frkX)(v)]+T(\rho(\delta(\frkX)(v))u)-[Tv,\delta(\frkX)(u)]-T(\rho(\delta(\frkX)(u))v)\\
~ &&-T(D_{\rho,\mu}(\frkX)[u,v]_T)+\Courant{\frkX,T[u,v]_T}\\
~ &\stackrel{\eqref{delta}}{=}&[Tu,TD_{\rho,\mu}(\frkX)v]-[Tu,\Courant{\frkX,Tv}]+T(\rho(T(D_{\rho,\mu}(\frkX)v))u)-T(\rho(\Courant{\frkX,Tv})u)\\
~ &&-[Tv,TD_{\rho,\mu}(\frkX)u]+[Tv,\Courant{\frkX,Tu}]-T(\rho(T(D_{\rho,\mu}(\frkX)u))v)+T(\rho(\Courant{\frkX,Tu})v)\\
~ &&-T(D_{\rho,\mu}(\frkX)[u,v]_T)+\Courant{\frkX,T[u,v]_T}\\
~ &\stackrel{\eqref{Ooperator1}}{=}&T(\rho(Tu)D_{\rho,\mu}(\frkX)v)-T(\rho(T(D_{\rho,\mu}(\frkX)v))u)-[Tu,\Courant{\frkX,Tv}]+T(\rho(T(D_{\rho,\mu}(\frkX)v))u)-T(\rho(\Courant{\frkX,Tv})u)\\
~ &&-T(\rho(Tv)D_{\rho,\mu}(\frkX)u)+T(\rho(T(D_{\rho,\mu}(\frkX)u))v)+[Tv,\Courant{\frkX,Tu}]-T(\rho(T(D_{\rho,\mu}(\frkX)u))v)+T(\rho(\Courant{\frkX,Tu})v)\\
~ &&-T(D_{\rho,\mu}(\frkX)(\rho(Tu)v-\rho(Tv)u))+\Courant{\frkX,[Tu,Tv]}\\
~ &\stackrel{\eqref{LY3},\eqref{RLYe}}{=}&0,\\
~ &&\\
~ &&\delta_{\rm II}^T\Big(\delta(\frkX)\Big)(u,v,w)\\
 ~ &\stackrel{\eqref{2cochain}}{=}&-\delta(\frkX)(\Courant{u,v,w}_T)+D_{\varrho,\varpi}(u,v)\Big(\delta (\frkX)w\Big)+\varpi(v,w)\Big(\delta (\frkX)u\Big)-\varpi(u,w)\Big(\delta (\frkX)v\Big)\\
 ~ &\stackrel{\eqref{repre1},\eqref{repre2}}{=}&\Courant{Tu,Tv,TD_{\rho,\mu}(\frkX)w-\Courant{\frkX,Tw}}+\Courant{TD_{\rho,\mu}(\frkX)u-\Courant{\frkX,Tu},Tv,Tw}\\
 ~ &&+\Courant{Tu,TD_{\rho,\mu}(\frkX)v-\Courant{\frkX,Tv},Tw}\\
 ~ &&-TD(\frkX)\Big(D_{\rho,\mu}(Tu,Tv)w+\mu(Tv,Tw)u-\mu(Tu,Tw)v\Big)\\
 ~ &&+\Courant{\frkX,T\big(D_{\rho,\mu}(Tu,Tv)w+\mu(Tv,Tw)u-\mu(Tu,Tw)v\big)}\\
 ~ &&-T\Big(D_{\rho,\mu}\big(TD_{\rho,\mu}(\frkX)u-\Courant{\frkX,Tu},Tv\big)w-D_{\rho,\mu}\big(TD_{\rho,\mu}(\frkX)v-\Courant{\frkX,Tv},Tu\big)w\Big)\\
 ~ &&-T\Big(\mu\big(TD_{\rho,\mu}(\frkX)v-\Courant{\frkX,Tv},Tw\big)u-\mu\big(TD_{\rho,\mu}(\frkX)u-\Courant{\frkX,Tu},Tw\big)v\Big)\\
 ~ &&-T\Big(\mu\big(Tv,TD_{\rho,\mu}(\frkX)w-\Courant{\frkX,Tw}\big)u-\mu\big(Tu,TD_{\rho,\mu}(\frkX)w-\Courant{\frkX,Tw}\big)v\Big)\\
 ~ &\stackrel{\eqref{Ooperator2}}{=}&\Courant{Tu,Tv,TD_{\rho,\mu}(\frkX)w}-\Courant{Tu,Tv,\Courant{\frkX,Tw}}+\Courant{TD_{\rho,\mu}(\frkX)u,Tv,Tw}\\
 ~ &&-\Courant{\Courant{\frkX,Tu},Tv,Tw}+\Courant{Tu,TD_{\rho,\mu}(\frkX)v,Tw}-\Courant{Tu,\Courant{\frkX,Tv},Tw}\\
 ~ &&-TD_{\rho,\mu}(\frkX)\Big(D_{\rho,\mu}(Tu,Tv)w+\mu(Tv,Tw)u-\mu(Tu,Tw)v\Big)+\Courant{\frkX,\Courant{Tu,Tv,Tw}}\\
 ~ &&-T\Big(D_{\rho,\mu}\big(TD_{\rho,\mu}(\frkX)u-\Courant{\frkX,Tu},Tv\big)w-D_{\rho,\mu}\big(TD_{\rho,\mu}(\frkX)v-\Courant{\frkX,Tv},Tu\big)w\Big)\\
 ~ &&-T\Big(\mu\big(TD_{\rho,\mu}(\frkX)v-\Courant{\frkX,Tv},Tw\big)u-\mu\big(TD_{\rho,\mu}(\frkX)u-\Courant{\frkX,Tu},Tw\big)v\Big)\\
 ~ &&-T\Big(\mu\big(Tv,\big(TD_{\rho,\mu}(\frkX)w-\Courant{\frkX,Tw}\big)u-\mu\big(Tu,TD_{\rho,\mu}(\frkX)w-\Courant{\frkX,Tw}\big)v\Big)\\
 ~ &\stackrel{\eqref{fundamental}}{=}&\Courant{Tu,Tv,TD_{\rho,\mu}(\frkX)w}+\Courant{TD_{\rho,\mu}(\frkX)u,Tv,Tw}+\Courant{Tu,TD_{\rho,\mu}(\frkX)v,Tw}\\
 ~ &&-TD_{\rho,\mu}(\frkX)\Big(D_{\rho,\mu}(Tu,Tv)w+\mu(Tv,Tw)u-\mu(Tu,Tw)v\Big)\\
 ~ &&-T\Big(D_{\rho,\mu}\big(TD_{\rho,\mu}(\frkX)u,Tv\big)w-D_{\rho,\mu}\big(\Courant{\frkX,Tu},Tv\big)w-D_{\rho,\mu}\big(TD_{\rho,\mu}(\frkX)v,Tu\big)w\Big)
 +D_{\rho,\mu}\big(\Courant{\frkX,Tv},Tu\big)w\Big)\\
 ~ &&-T\Big(\mu\big(TD_{\rho,\mu}(\frkX)v,Tw\big)u-\mu\big(\Courant{\frkX,Tv},Tw\big)u-\mu\big(TD_{\rho,\mu}(\frkX)u,Tw\big)v\Big)+\mu\big(\Courant{\frkX,Tu},Tw\big)v\Big)\\
 ~ &&-T\Big(\mu\big(Tv,TD_{\rho,\mu}(\frkX)w\big)u-\mu\big(Tv,\Courant{\frkX,Tw}\big)u-\mu\big(Tu,TD_{\rho,\mu}(\frkX)w\big)v+\mu\big(Tu,\Courant{\frkX,Tw}\big)v\Big)\\
 ~ &\stackrel{\eqref{Ooperator2}}{=}&T\Big(D_{\rho,\mu}(Tu,Tv)D_{\rho,\mu}(\frkX)w+\mu(Tv,TD_{\rho,\mu}(\frkX)w)u-\mu(Tu,TD_{\rho,\mu}(\frkX)w)v\Big)\\
 ~ &&+T\Big(D_{\rho,\mu}(TD_{\rho,\mu}(\frkX)u,Tv)w+\mu(Tv,Tw)D_{\rho,\mu}(\frkX)u-\mu(TD_{\rho,\mu}(\frkX)u,Tw)v\Big)\\
 ~ &&+T\Big(D_{\rho,\mu}(Tu,TD_{\rho,\mu}(\frkX)v)w+\mu(TD_{\rho,\mu}(\frkX)v,Tw)u-\mu(Tu,Tw)D_{\rho,\mu}(\frkX)v\Big)\\
 ~ &&-TD_{\rho,\mu}(\frkX)\Big(D_{\rho,\mu}(Tu,Tv)w+\mu(Tv,Tw)u-\mu(Tu,Tw)v\Big)\\
 ~ &&-T\Big(D_{\rho,\mu}\big(TD_{\rho,\mu}(\frkX)u,Tv\big)w-D_{\rho,\mu}\big(\Courant{\frkX,Tu},Tv\big)w-D_{\rho,\mu}\big(TD_{\rho,\mu}(\frkX)v,Tu\big)w
 +D_{\rho,\mu}\big(\Courant{\frkX,Tv},Tu\big)w\Big)\\
 ~ &&-T\Big(\mu\big(TD_{\rho,\mu}(\frkX)v,Tw\big)u-\mu\big(\Courant{\frkX,Tv},Tw\big)u-\mu\big(TD_{\rho,\mu}(\frkX)u,Tw\big)v)+\mu\big(\Courant{\frkX,Tu},Tw\big)v\Big)\\
 ~ &&-T\Big(\mu\big(Tv,TD_{\rho,\mu}(\frkX)w\big)u-\mu\big(Tv,\Courant{\frkX,Tw}\big)u-\mu\big(Tu,TD_{\rho,\mu}(\frkX)w\big)v+\mu\big(Tu,\Courant{\frkX,Tw}\big)v\Big)\\
 ~ &=&T\Big(D_{\rho,\mu}(Tu,Tv)D_{\rho,\mu}(\frkX)w+\mu(Tv,Tw)D_{\rho,\mu}(\frkX)u-\mu(Tu,Tw)D_{\rho,\mu}(\frkX)v\Big)\\
 ~ &&-TD_{\rho,\mu}(\frkX)\Big(D_{\rho,\mu}(Tu,Tv)w+\mu(Tv,Tw)u-\mu(Tu,Tw)v\Big)\\
 ~ &&+T\Big(D_{\rho,\mu}(\Courant{\frkX,Tu},Tv)w-D_{\rho,\mu}(\Courant{\frkX,Tv},Tu)w+\mu(\Courant{\frkX,Tv},Tw)u\\
 ~ &&-\mu(\Courant{\frkX,Tu},Tw)v+\mu(Tv,\Courant{\frkX,Tw})u-\mu(Tu,\Courant{\frkX,Tw})v\Big)\\
 &\stackrel{\eqref{RLY5},\eqref{RLY5a}}{=}&0.
\end{eqnarray*}
}
This finishes the proof.
\end{proof}

So far, we have constructed a new complex starting from $0$-cochains, whose cohomology is defined to be that of relative Rota-Baxter operators.
\begin{defi}\label{cohomology}
Let $T$ be a relative Rota-Baxter operator on a Lie-Yamaguti algebra $(\g,[\cdot,\cdot],\Courant{\cdot,\cdot,\cdot})$ with  respect to a representation $(V;\rho,\mu)$. Define the set of $n$-cochains by
\begin{eqnarray}
\huaC^n_T(V,\g):=
\begin{cases}
C^n_{\rm LieY}(V,\g),&n\geqslant 1,\\
\wedge^2\g,&n=0.
\end{cases}
\end{eqnarray}
Define the coboundary map $\de:\huaC^n_{T}(V,\g)\to \huaC^{n+1}_{T}(V,\g)$ by
\begin{eqnarray}
\de:=
\begin{cases}
\delta^T=(\delta^T_{\rm I},\delta^T_{\rm II}),&n\geqslant 1,\\
\delta,&n=0.
\end{cases}
\end{eqnarray}
Thus we obtain a well-defined cochain complex $(\huaC_T^\bullet(V,\g)=\bigoplus\limits_{n=0}^\infty\huaC_T^n(V,\g),\de)$, and we call it the {\bf cohomology of relative Rota-Baxter operator $T$} on the Lie-Yamaguti algebra $(\g,[\cdot,\cdot],\Courant{\cdot,\cdot,\cdot})$ with respect to the representation $(V;\rho,\mu)$. Denote the set of $n$-cocycles and $n$-coboundaries by $\huaZ^n(V,\g)$ and $\huaB^n(V,\g)$ respectively. The {\bf $n$-th cohomology group of relative Rota-Baxter operator $T$} will be taken to be
\begin{eqnarray}
\huaH^n_T(V,\g):=\huaZ^n_T(V,\g)/\huaB^n_T(V,\g), \quad n\geqslant 1.
\end{eqnarray}
\end{defi}

\section{Deformatons of relative Rota-Baxter operators on Lie-Yamaguti algebras}

In this section, we will use the cohomology theory constructed in the former section to characterize deformations of relative Rota-Baxter operators on Lie-Yamaguti algebras.

\subsection{Linear deformations of relative Rota-Baxter operators on Lie-Yamaguti algebras}

In this subsection, we aim to linear deformations of relative Rota-Baxter operators on Lie-Yamaguti algebras, and we show that the infinitesimals of two equivalent linear deformations of a relative Rota-Baxter operator on Lie-Yamaguti algebra are in the same cohomology class of the first cohomology group.
\begin{defi}
Let $T$ and $T'$ be two relative Rota-Baxter operators on a Lie-Yamaguti algebra $(\g,[\cdot,\cdot],\Courant{\cdot,\cdot,\cdot})$ with respect to a representation $(V;\rho,\mu)$. A {\bf homomorphism} from $T'$ to $T$ is a pair $(\phi_\g,\phi_V)$, where $\phi_\g: \g \to \g$ is a \LYA~ homomorphism and $\phi_V: V \to V$  is a linear map satisfying
\begin{eqnarray}
\label{homo1}T\circ \phi_V&=&\phi_\g\circ T'\\
~\label{homo2}\phi_V(\rho(x)v)&=&\rho(\phi_\g(x))\phi_V(v),\\
~\label{homo3}\phi_V\mu(x,y)(v)&=&\mu(\phi_\g(x),\phi_\g(y))(\phi_V(v)), \quad \forall x,y \in \g,~v\in V.
\end{eqnarray}
In particular, if $\phi_\g$ and $\phi_V$ are invertible, then $(\phi_\g,\phi_V)$ is called an {\bf isomorphism} from $T'$ to $T$.
\end{defi}

By a direct computation, we have the following lemma.
\begin{lem}
Let $T$ and $T'$ be two relative Rota-Baxter operators on a Lie-Yamaguti algebra $(\g,[\cdot,\cdot],\Courant{\cdot,\cdot,\cdot})$ with respect to a representation $(V;\rho,\mu)$, and $(\phi_\g,\phi_V)$ a homomorphism from $T'$ to $T$, then we have
\begin{eqnarray}
\label{homo4}\phi_VD_{\rho,\mu}(x,y)(v)&=&D_{\rho,\mu}(\phi_\g(x),\phi_\g(y))(\phi_V(v)), \quad \forall x,y \in \g,~v\in V.
\end{eqnarray}
\end{lem}

In the sequel, we would write $D$ for $D_{\rho,\mu}$ without ambiguity.

\begin{pro}
Let $T$ and $T'$ be two relative Rota-Baxter operators on a Lie-Yamaguti algebra $(\g,[\cdot,\cdot],\Courant{\cdot,\cdot,\cdot})$ with respect to a representation $(V;\rho,\mu)$ and $(\phi_\g,\phi_V)$ a homomorphism from $T'$ to $T$. Then $\phi_V$ is a homomorphism from pre-Lie-Yamaguti algebra from $(V,*_{T'},\{\cdot,\cdot,\cdot\}_{T'})$ to $(V,*_T,\{\cdot,\cdot,\cdot\}_{T})$.
\end{pro}
\begin{proof}
For all $u,v,w\in V$, we have
\begin{eqnarray*}
\phi_V(u*_{T'}v)&=&\phi_V(\rho(T'u)v)=\rho(\phi_\g(T'u)\phi_V(v))\\
~ &=&\rho(T(\phi_V(u))\phi_V(v))=\phi_V(u)*_T\phi_V(v),\\
\phi_V(\{u,v,w\}_{T'})&=&\phi_V(\mu(T'v,T'w)u)=\mu(\phi_\g(T'v),\phi_\g(T'w))(\phi_V(u))\\
~ &=&\mu(T(\phi_V(v)),T(\phi_V(w)))(\phi_V(u))=\{\phi_V(u),\phi_V(v),\phi_V(w)\}_{T}.
\end{eqnarray*}
This finishes the proof.
\end{proof}

The notion of linear deformations of relative Rota-Baxter operators is given as follows.
\begin{defi}
Let $(\g,[\cdot,\cdot],\Courant{\cdot,\cdot,\cdot})$ be a \LYA, and $(V;\rho,\mu)$ a representation of $\g$. Suppose that $T,~\frkT:V \to \g$ are two linear maps, where $T$ is a relative Rota-Baxter operator on $\g$ with respect to $(V;\rho,\mu)$. If $T_t=T+t\frkT$ are still relative Rota-Baxter operators on $\g$ with respect to $(V;\rho,\mu)$ for all $t$, we say that $\frkT$ generates a {\bf linear deformation} of the relative Rota-Baxter operator $T$.
\end{defi}

It is easy to see that $\frkT$ generates a linear deformation of the relative Rota-Baxter operator $T$ if and only if
\begin{eqnarray}
~\label{cocy1}~ &&[\frkT u,Tv]+[Tu,\frkT v]=T\big(\rho(\frkT u)v-\rho(\frkT v)u\big)+\frkT\big(\rho(Tu)v-\rho(Tv)u\big),\\
~\label{Ooper1}&&[\frkT u,\frkT v]=\frkT \big(\rho(\frkT u)v-\rho(\frkT v)u\big),\\
~ \label{cocy}&&\Courant{Tu,Tv,\frkT w}+\Courant{Tu,\frkT v,Tw}+\Courant{\frkT u,Tv,Tw}\\
~ \nonumber&=&\frkT\Big(D(Tu,Tv)w+\mu(Tv,Tw)v-\mu(Tu,Tw)v\Big)\\
~ &&\nonumber+T\Big(D(Tu,\frkT v)w+D(\frkT u,Tv)w+\mu(Tv,\frkT w)u\\
~ &&\nonumber+\mu(\frkT v,Tw)u-\mu(Tu,\frkT w)v-\mu(\frkT u,Tw)v\Big),\\
~ &&\Courant{\frkT u,\frkT v, Tw}+\Courant{Tu,\frkT v,\frkT w}+\Courant{\frkT u,Tv,\frkT w}\\
~ \nonumber&=&T\Big(D(\frkI u,\frkI v)w+\mu(\frkI v,\frkI w)v-\mu(\frkI u,\frkI w)v\Big)\\
~ &&\nonumber+\frkT\Big(D(Tu,\frkT v)w+D(\frkT u,Tv)w+\mu(Tv,\frkT w)u\\
~ &&\nonumber+\mu(\frkT v,Tw)u-\mu(Tu,\frkT w)v-\mu(\frkT u,Tw)v\Big),\\
~ \label{Oopera}&&\Courant{\frkT u,\frkT v, \frkT w}=\frkT \Big(D(\frkT u,\frkT v)w+\mu(\frkT v,\frkT w)u-\mu(\frkT u,\frkT w)v\Big).
\end{eqnarray}

The following can be deduced easily.
\begin{itemize}
\item [\rm (i)] \eqref{cocy1} and \eqref{cocy} means that $\frkT\in\huaC^1(V,\g)$ is a $1$-cocycle of $\delta^T$.
\item [\rm (ii)] \eqref{Ooper1} and \eqref{Oopera} means that $\frkT$ is a relative Rota-Baxter operator on the Lie-Yamaguti algebra $(\g,[\cdot,\cdot],\Courant{\cdot,\cdot,\cdot})$ with respect to the representation $(V;\rho,\mu)$.
\end{itemize}

Let $(A,*,\{\cdot,\cdot,\cdot\})$ be a pre-Lie-Yamaguti algebra, and let $\phi:\otimes^2A\to A$ and $\omega_1,~\omega_2:\otimes^3A \to A$ be linear maps. If the linear operations $(*_t,\{\cdot,\cdot,\cdot\}_t)$ defined by
\begin{eqnarray}
x*_ty&=&x*y+t\phi(x,y),\\
\{x,y,z\}_t&=&\{x,y,z\}+t\omega_1(x,y,z)+t^2\omega_2(x,y,z),\quad \forall x,y,z \in A,
\end{eqnarray}
are still pre-Lie-Yamaguti algebra structures for all $t$, we say that $(\phi,\omega_1,\omega_2)$ generates a linear deformation of the pre-Lie-Yamaguti algebra $A$.

Thanks to the relationship between relative Rota-Baxter operators on \LYA s and pre-\LYA, we have the following proposition.
\begin{pro}
If $\frkT$ generates a linear deformation of the relative Rota-Baxter operator $T$ on a Lie-Yamaguti algebra $(\g,[\cdot,\cdot],\Courant{\cdot,\cdot,\cdot})$ with respect to a representation $(V;\rho,\mu)$, then the triple $(\phi_\frkT,\omega_\frkT^1,\omega_\frkT^2)$ generates a linear deformation of the underlying pre-Lie-Yamaguti algebra $(V,*,\{\cdot,\cdot,\cdot\}_T)$, where
\begin{eqnarray}
\phi_\frkT(u,v)&=&\rho(\frkT(u))v,\\
\omega_\frkT^1(u,v,w)&=&\mu(Tv,\frkT w)u+\mu(\frkT v,Tw)u,\\
\omega_\frkT^2(u,v,w)&=&\mu(\frkT v,\frkT w )u, \quad \forall u,v, w \in V.
\end{eqnarray}
\end{pro}
\begin{proof}
Denote by $(*_t,\{\cdot,\cdot,\cdot\}_t)$ the corresponding pre-Lie-Yamaguti algebra structure induced by the relative Rota-Baxter operator $T_t:=T+t\frkT$. Indeed, for all $u,v,w \in V$, we have that
\begin{eqnarray*}
u*_tv&=&\rho((T+t\frkT)u)v=\rho(Tu)v+t\rho(\frkT u)v=u*_Tv+t\phi_\frkT(u,v),\\
\{u,v,w\}_t&=&\mu((T+t\frkT)v,(T+t\frkT)w)u\\
~ &=&\mu(Tv,Tw)u+t\Big(\mu(Tv,\frkT w)u+\mu(\frkT v,Tw)u\Big)+t^2\mu(\frkT v,\frkT w )u\\
~ &=&\{u,v,w\}_T+t\omega_\frkT^1(u,v,w)+t^2\omega_\frkT^2(u,v,w).
\end{eqnarray*}
This finishes the proof.
\end{proof}

\begin{defi}
Let $T:V\to \g$ be a relative Rota-Baxter operator on a Lie-Yamaguti algebra $(\g,[\cdot,\cdot],\Courant{\cdot,\cdot,\cdot})$ with respect to a representation $(V;\rho,\mu)$.
\begin{itemize}
\item [\rm (i)]  Two linear deformations $T_t^1=T+t\frkT_1$ and $T_t^2=T+t\frkT_2$ are said to be {\bf equivalent} if there exists an element $\frkX \in \wedge^2\g$ such that $({\Id}_\g+t\frkL_{\frkX},{\Id}_V+tD(\frkX))$ is a homomorphism from $T_t^2$ to $T_t^1$.
\item [\rm (ii)] A linear deformation $T_t=T+t\frkT$ of a relative Rota-Baxter operator $T$ is said to be {\bf trivial} if there exists an element $\frkX \in \wedge^2\g$ such that $({\Id}_\g+t\frkL_{\frkX},{\Id}_V+tD(\frkX))$ is a homomorphism from $T_t$ to $T$.
    \end{itemize}
\end{defi}

Let $({\Id}_\g+t\frkL_{\frkX},{\Id}_V+tD(\frkX))$ be a homomorphism from $T_t^2$ to $T_t^1$. Then ${\Id}_\g+t\frkL_{\frkX}$ is a Lie-Yamaguti algebra homomorphism of $\g$, i.e., the following equalities hold:
\begin{eqnarray}
~\label{Nije} [\Courant{\frkX,x},\Courant{\frkX,y}]&=&0,\\
~ \label{Nij1}\Courant{\Courant{\frkX,x},\Courant{\frkX,y},z}+\Courant{\Courant{\frkX,x},y,\Courant{\frkX,z}}+\Courant{x,\Courant{\frkX,y},\Courant{\frkX,z}}&=&0,\\
~ \label{Nij2}\Courant{\Courant{\frkX,x},\Courant{\frkX,y},\Courant{\frkX,z}}&=&0.
\end{eqnarray}

By $T_t^1\big(({\Id_V}+tD(\frkX))v\big)=\big({\Id}_\g+t\frkL_{\frkX}\big)T_t^2(v)$, we have
\begin{eqnarray}
\label{cocycle}(\frkT_2-\frkT_1)(v)&=&T\Big(D(\frkX)v\Big)-\Courant{\frkX,Tv},\\
\frkT_1\Big(D(\frkX)v\Big)&=&\Courant{\frkX,\frkT_2(v)}.
\end{eqnarray}

By $\Big({\Id}_V+tD(\frkX)\Big)(\rho(x)v)=\rho\Big(({\Id}_\g+t\frkL_\frkX)(x)\Big)({\Id}_V+tD(\frkX))(v)$, we have
\begin{eqnarray}
\rho(\Courant{\frkX,x})D(\frkX)=0.
\end{eqnarray}

Finally, by $\Big({\Id}_V+tD(\frkX)\Big)\mu(z,w)v=\mu\Big(({\Id}_\g+t\frkL_{\frkX})z,({\Id}_\g+t\frkL_{\frkX})w\Big)({\Id}_V+tD(\frkX))v$, we have
\begin{eqnarray}
\label{Nij3}\mu(z,\Courant{\frkX,w})D(\frkX)+\mu(\Courant{\frkX,z},w)D(\frkX)+\mu(\Courant{\frkX,z},\Courant{\frkX,w})&=&0,\\
~\label{Nij4}\mu(\Courant{\frkX,z},\Courant{\frkX,w})D(\frkX)&=&0.
\end{eqnarray}

Note that \eqref{cocycle} means that there exists $\frkX\in \wedge^2\g$, such that $\frkT_2-\frkT_1=\delta(\frkX)$. Thus we have the following key conclusion in this section.
\begin{thm}\label{thm1}
Let $T:V\to \g$ be a relative Rota-Baxter operator on a Lie-Yamaguti algebra $(\g,[\cdot,\cdot],\Courant{\cdot,\cdot,\cdot})$ with respect to a representation $(V;\rho,\mu)$. If two linear deformations $T_t^1=T+t\frkT_1$ and $T_t^2=T+t\frkT_2$ of $T$ are equivalent, then $\frkT_1$ and $\frkT_2$ are in the same class of the cohomology group $\huaH^1_T(V,\g)$.
\end{thm}

\begin{defi}
Let $T:V\to \g$ be a relative Rota-Baxter operator on a Lie-Yamaguti algebra $(\g,[\cdot,\cdot],\Courant{\cdot,\cdot,\cdot})$ with respect to a representation $(V;\rho,\mu)$. An element $\frkX\in \wedge^2\g$ is called a {\bf Nijenhuis element} with respect to $T$ if $\frkX$ satisfies \eqref{Nije}-\eqref{Nij2}, \eqref{Nij3}, \eqref{Nij4} and the following equation
\begin{eqnarray}
\Courant{\frkX,T(D(\frkX)v)-\Courant{\frkX,Tv}}=0, \quad \forall v \in V. \label{Nij5}
\end{eqnarray}
We denote by $\Nij(T)$ the set of Nijenhuis elements  with respect to $T$.
\end{defi}

It is obvious that a trivial deformation of a relative Rota-Baxter operator on a Lie-Yamaguti algebra gives rise to a Nijenhuis element. Indeed, the converse is also true.
\begin{thm}\label{Nijenhuis}
Let $T:V\to \g$ be a relative Rota-Baxter operator on a Lie-Yamaguti algebra $(\g,[\cdot,\cdot],\Courant{\cdot,\cdot,\cdot})$ with respect to a representation $(V;\rho,\mu)$. Then for any Nijenhuis element $\frkX\in \wedge^2 \g$, $T_t:=T+t\frkT$ with $\frkT:=\delta(\frkX)$ is a trivial linear deformation of the relative Rota-Baxter operator $T$.
\end{thm}
One way to prove the above Theorem is to check that $\frkT$ satisfies Eqs. \eqref{cocy1}-\eqref{Oopera}. Now, however, we decide to prove it another way. Let us give a lemma first, which follows from a direct computation.

\begin{lem}\label{lemma}
Let $T:V\to \g$ be a relative Rota-Baxter operator on a Lie-Yamaguti algebra $(\g,[\cdot,\cdot],\Courant{\cdot,\cdot,\cdot})$ with respect to a representation $(V;\rho,\mu)$. Let $\phi_\g:\g \to \g$ be a \LYA ~isomorphism and $\phi_V:V\to V$ an isomorphism between vector spaces such that Eqs. \eqref{homo2}-\eqref{homo3} hold. Then $\phi_\g^{-1}\circ T\circ \phi_V:V\to \g$ is a relative Rota-Baxter operator on the Lie-Yamaguti algebra $(\g,[\cdot,\cdot],\Courant{\cdot,\cdot,\cdot})$ with respect to the representation $(V;\rho,\mu)$.
\end{lem}

\noindent
\emph{Proof of Theorem {\rm\ref{Nijenhuis}:}} For any Nijenhuis element $\frkX\in \Nij(T) \subset \wedge^2\g$, we define
\begin{eqnarray}
\frkT=\delta\frkX.
\end{eqnarray}
Since $\frkX$ is a Nijenhuis element, for all $t$, $T_t=T+t\frkT$ satisfies
\begin{eqnarray*}
({\Id}_\g+t\frkL_\frkX)\circ T_t&=&T\circ({\Id}_V+tD(\frkX)),\\
~({\Id}_V+tD(\frkX))\rho(x)v&=&\rho(({\Id}_\g+t\frkL_\frkX)(x))({\Id}_V+tD(\frkX))(v),\\
~({\Id}_V+tD(\frkX))\mu(x,y)v&=&\mu({\Id}_\g+t\frkL_\frkX)(x),{\Id}_\g+t\frkL_\frkX)(y))({\Id}_V+tD(\frkX))(v),\quad \forall x,y \in\g, ~v \in V.
\end{eqnarray*}
For $t$ sufficiently small, we see that ${\Id}_\g+t\frkL_\frkX$ is a Lie-Yamaguti algebra isomorphism and that ${\Id}_V+tD(\frkX)$ is an isomorphism between vector spaces. Thus, we have
$$T_t=({\Id}_\g+t\frkL_\frkX)^{-1}\circ T\circ ({\Id}_V+tD(\frkX)).$$
By Lemma \ref{lemma}, we see that $T_t$ is a relative Rota-Baxter operator on the Lie-Yamaguti algebra $(\g,[\cdot,\cdot],\Courant{\cdot,\cdot,\cdot})$ with respect to $(V;\rho,\mu)$ for $t$ sufficiently small. Thus $\frkT=\delta\frkX$ satisfies Eqs. \eqref{cocy1}-\eqref{Oopera}. Therefore, $T_t$ is a relative Rota-Baxter operator for all $t$, which implies that $\frkT$ generates a liner deformation of $T$. It is easy to see that this deformation is trivial.
\qed

\vspace{2mm}

In the sequel, we are going to consider deformations of Rota-Baxter operators on Lie-Yamaguti algebras which is a special case of relative Rota-Baxter operators. Thus the conclusions of Rota-Baxter operators are direct corollaries of the former. We will also give some examples of Nijenhuis elements of Rota-Baxetr operators at the end of this section. Recall that a Rota-Baxter operator on a Lie-Yamaguti algebra $(\g,[\cdot,\cdot],\Courant{\cdot,\cdot,\cdot})$ is a relative Rota-Baxter operator with respect to the adjoint representation $(\g;\ad,\frkR)$, i.e., a linear map $R:\g \to \g$ satisfying
 \begin{eqnarray}
 [Rx,Ry]&=&R\Big([Rx,y]-[Ry,x]\Big),\\
 \Courant{Rx,Ry,Rz}&=&R\Big(\Courant{Rx,Ry,z}+\Courant{x,Ry,Rz}+\Courant{Rx,y,Rz}\Big),\quad \forall x,y,z \in \g.
 \end{eqnarray}

 Let $R$ be a Rota-Baxter operator on a Lie-Yamaguti algebra $(\g,[\cdot,\cdot],\Courant{\cdot,\cdot,\cdot})$, then the induced pre-Lie-Yamaguti algebra structure on $\g$ is given by $x*_Ry=[Rx,y],~\{x,y,z\}_R=\Courant{x,Ry,Rz}$ for all $x,y,z \in \g$, and its sub-adjacent Lie-Yamaguti algebra structure is given by $[x,y]_R=[Rx,y]-[Ry,x],~\Courant{x,y,z}_R=\Courant{Rx,Ry,z}+\Courant{x,Ry,Rz}+\Courant{Rx,y,Rz}$ for all $x,y,z \in \g$.

 \begin{pro}
 Let $R$ be a Rota-Baxter operator on a Lie-Yamaguti algebra $(\g,[\cdot,\cdot],\Courant{\cdot,\cdot,\cdot})$, then $\varrho:\g \to \gl(\g),~ \varpi:\otimes^2\g \to \gl(\g)$ given by
 \begin{eqnarray}
 \varrho(x)y&=&[Rx,y]-R\Big([x,y]\Big),\\
 \varpi(x,y)z&=&\Courant{z,Rx,Ry}-R\Big(\Courant{z,Rx,y}-\Courant{x,Ry,z}\Big), \quad \forall x,y,z \in \g
 \end{eqnarray}
 forms a representation of the sub-adjacent Lie-Yamaguti algebra $(\g,[\cdot,\cdot]_R,\Courant{\cdot,\cdot,\cdot}_R)$ on itself, where
 \begin{eqnarray}
 D_{\varrho,\varpi}(x,y)z=\Courant{Rx,Ry,z}-R\Big(\Courant{x,Ry,z}-\Courant{y,Rx,z}\Big),\quad \forall x,y,z \in \g.
 \end{eqnarray}
 \end{pro}

 \begin{defi}
 Let $R$ be a Rota-Baxter operator on a Lie-Yamaguti algebra $(\g,[\cdot,\cdot],\Courant{\cdot,\cdot,\cdot})$. Then the cohomology of the cochain complex $(\oplus_{k\geqslant0}\huaC^k_R(\g,\g),\de)~(k\geqslant0)$ given in Definition \ref{cohomology}, is called the {\bf cohomology of Rota-Baxter operator $R$}. The resulting $n$-th cohomology group is denoted by $\huaH^n_R(V,\g)$.
 \end{defi}

 \begin{defi}
 Let $R$ be a Rota-Baxter operator on a Lie-Yamaguti algebra $(\g,[\cdot,\cdot],\Courant{\cdot,\cdot,\cdot})$.
 \begin{itemize}
 \item[\rm (i)] Let $\huaR:\g \to \g$ be a linear map. If for all $t\in \mathbb K$, $R_t:=R+t\huaR$ is a Rota-Baxter operator on $\g$, we say that $\huaR$ generates a {\bf linear deformation} of $R$.
     \item[\rm (ii)] Let $R_t^1:=R+t\huaR_1$ and $R_t^2:=R+t\huaR_2$ be two linear deformations of $R$. They are said to be {\bf equivalent} if there exists an element $\frkX\in \wedge^2\g$ such that $({\Id}_\g+t\frkL_\frkX,{\Id}_V+tD(\frkX))$ is a homomorphism from $R_t^2$ to $R_t^1$. In particular, a deformation $R_t=R+t\huaR$ is said to be {\bf trivial} if there exists an element $\frkX\in \wedge^2\g$ such that $({\Id}_\g+t\frkL_\frkX,{\Id}_V+tD(\frkX))$ is a homomorphism from $R_t$ to $R$.
 \end{itemize}
 \end{defi}

 \begin{pro}
  Let $R$ be a Rota-Baxter operator on a Lie-Yamaguti algebra $(\g,[\cdot,\cdot],\Courant{\cdot,\cdot,\cdot})$. If $\huaR$ generates a linear deformation of $R$, then $\huaR$ is a $1$-cocycle. Moreover, if two linear deformations of $R$ generated by $\huaR_1$ and $\huaR_2$ respectively are equivalent, then $\huaR_1$ and $\huaR_2$ are in the same cohomology class of $\huaH_R^1(V,\g)$.
 \end{pro}

 \begin{defi}
 Let $R$ be a Rota-Baxter operator on a Lie-Yamaguti algebra $(\g,[\cdot,\cdot],\Courant{\cdot,\cdot,\cdot})$. An element $\frkX\in \wedge^2\g$ is a {\bf Nijenhuis element} if it satisfies the conditions \eqref{Nije}-\eqref{Nij2}, and the following condition
 $$\Courant{\frkX,R(\Courant{\frkX,y})-\Courant{\frkX,Ry}}, \quad \forall y \in \g.$$
 \end{defi}

At the end of this subsection, we will give two examples of Nijenhuis elements associated to Rota-Baxter operators.
 \begin{ex}
{\rm Let $(\g,[\cdot,\cdot],\Courant{\cdot,\cdot,\cdot})$ be a 2-dimensional Lie-Yamaguti algebra, whose nontrivial brackets are given by, with respect to a basis $\{e_1,e_2\}$
$$[e_1,e_2]=e_1,~\quad \Courant{e_1,e_2,e_2}=e_1.$$
Moreover,
$$R=
\begin{pmatrix}
0 & a\\
0 & b
\end{pmatrix}$$
is a Rota-Baxter operator on $\g$. The by a direct computation, any element in $\wedge^2\g$ is a Nijenhuis element of $R$.}
 \end{ex}

 \begin{ex}
 {\rm Let $\g$ be a 4-dimensional Lie-Yamaguti algebra with a basis $\{e_1,e_2,e_3,e_4\}$ defined by
 $$[e_1,e_2]=2e_4,~\quad \Courant{e_1,e_2,e_1}=e_4.$$
 And
 $$R=
 \begin{pmatrix}
 0 & a_{12}& 0 & 0 \\
 0 & 0& 0 & 0\\
 a_{31} &a_{32} & a_{33} & a_{34}\\
 a_{41} &a_{42} & a_{43} & a_{44}
 \end{pmatrix}$$
 is a Rota-Baxter operator on $\g$. Then any element in $\wedge^2\g$ is a Nijenhuis element of $R$. In particular,
 \begin{eqnarray*}
 \frkX_1=e_1\wedge e_2,\quad \frkX_2=e_1\wedge e_3,\quad \frkX_3=e_1\wedge e_4,\\
 \frkX_4=e_2\wedge e_3,\quad \frkX_5=e_2\wedge e_4,\quad \frkX_6=e_3\wedge e_4,
 \end{eqnarray*}
 are all Nijenhuis elements of $R$.}
 \end{ex}

\subsection{Formal deformations of relative Rota-Baxter operators on Lie-Yamaguti algebras}
In this subsection, we study formal deformations of relative Rota-Baxter operators on Lie-Yamaguti algebras. Let $\mathbb K[[t]]$ be a ring of power series of one variable $t$. For any linear vector space $V$, $V[[t]]$ denotes the vector space of formal power series  of $t$ with the coefficients in $V$. If $(\g,[\cdot,\cdot],\Courant{\cdot,\cdot,\cdot})$ is a Lie-Yamaguti algebra, then there is a Lie-Yamaguti algebra structure over the ring $\mathbb K[[t]]$ on $\g[[t]]]$ given by
\begin{eqnarray}
\Big[\sum_{i=0}^\infty x_it^i,\sum_{j=0}^\infty y_jt^j\Big]&=&\sum_{s=0}^\infty\sum_{i+j=s} \Big[x_i,y_j\Big]t^s,\\
~\small{\Courant{\sum_{i=0}^\infty x_it^i,\sum_{j=0}^\infty y_jt^j,\sum_{k=0}^\infty z_kt^k}}&=&\sum_{s=0}^\infty\sum_{i+j+k=s}\Courant{x_i,y_j,z_k}t^s,\quad \forall x_i,y_j,z_k\in \g.
\end{eqnarray}
For any representation $(V;\rho,\mu)$ of a Lie-Yamaguti algebra $(\g,[\cdot,\cdot],\Courant{\cdot,\cdot,\cdot})$, there is a nature representation of the Lie-Yamaguti algebra $\g[[t]]$ on the  $\mathbb K[[t]]$-module $V[[t]]$ given by
\begin{eqnarray}
\rho\Big(\sum_{i=0}^\infty x_it^i\Big)\Big(\sum_{k=0}^\infty v_kt^k\Big)&=&\sum_{s=0}^\infty\sum_{i+k=s}\rho(x_i)v_kt^s,\\
\mu\Big(\sum_{i=0}^\infty x_it^i,\sum_{j=0}^\infty y_jt^j\Big)\Big(\sum_{k=0}^\infty v_kt^k\Big)&=&\sum_{s=0}^\infty\sum_{i+j+k=s}\mu(x_i,x_j)v_kt^s, \quad \forall x_i,y_j\in \g,~v_k\in V.
\end{eqnarray}

Let $T$ be a relative Rota-Baxter operator on a Lie-Yamaguti algebra $(\g,[\cdot,\cdot],\Courant{\cdot,\cdot,\cdot})$ with respect to a representation $(V;\rho,\mu)$. Consider the power series
\begin{eqnarray}
\label{defor:T}T_t=\sum_{i=0}^\infty\frkT_it^i,\quad \frkT_i\in \Hom(V,\g),
\end{eqnarray}
that is, $T_t\in \Hom_{\mathbb K}(V,\g)[[t]]=\Hom_{\mathbb K}(V,\g[[t]])$.

\begin{defi}
Let $T$ be a relative Rota-Baxter operator on a Lie-Yamaguti algebra $(\g,[\cdot,\cdot],\Courant{\cdot,\cdot,\cdot})$ with respect to a representation $(V;\rho,\mu)$. Suppose that $T_t$ is given by \eqref{defor:T}, where $\frkT_0=T$, and $T_t$ also satisfies
\begin{eqnarray}
\label{deforO1}[T_tu,T_tv]&=&T_t\Big(\rho(T_tu)v-\rho(T_tv)u\Big),\\
\label{deforO2}\Courant{T_tu,T_tv,T_tw}&=&T_t\Big(D_{\rho,\mu}(T_tu,T_tv)w+\mu(T_tv,T_tw)u-\mu(T_tu,T_tw)v\Big), \quad \forall u,v,w \in V.
\end{eqnarray}
We say that $T_t$ is a {\bf formal deformation} of $T$.
\end{defi}

 Recall that a formal deformation of a \LYA ~$(\g,\br,\ltp))$ is a pair of power series $f_t=\sum_{i=0}^\infty f_it^i$ and $g_t=\sum_{j=0}^\infty g_jt^j$, where $f_0=[\cdot,\cdot]$ and $g_0=\Courant{\cdot,\cdot,\cdot}$, and $(f_t,g_t)$ defines a \LYA ~structure on $\g[[t]]$ (\cite{L.CHEN}). Based on the relationship between the relative Rota-Baxter operators and the pre-Lie-Yamaguti algebras, we have the following proposition.
 \begin{pro}
 If $T_t=\sum_{i=0}^\infty\frkT_it^i$ is a formal deformation of a relative Rota-Baxter operator $T$ on a Lie-Yamaguti algebra $(\g,[\cdot,\cdot],\Courant{\cdot,\cdot,\cdot})$ with respect to $(V;\rho,\mu)$. Then $([\cdot,\cdot]_{T_t},\Courant{\cdot,\cdot,\cdot}_{T_t})$ defined by
 \begin{eqnarray}
 [u,v]_{T_t}&=&\sum_{i=0}^\infty \Big(\rho(\frkT_iu)v-\rho(\frkT_iv\Big)t^i,\\
 \Courant{u,v,w}_{T_t}&=&\sum_{k=0}^\infty\sum_{i+j=k}\Big(D_{\rho,\mu}(\frkT_iu,\frkT_jv)w+\mu(\frkT_iv,\frkT_jw)u-\mu(\frkT_iu,\frkT_jw)v\Big)t^k, \quad u,v,w\in V,
 \end{eqnarray}
 is a formal deformation of the Lie-Yamaguti algebra $(V,[\cdot,\cdot]_T,\Courant{\cdot,\cdot,\cdot}_T)$.
 \end{pro}
 Substituting the Eq. \eqref{defor:T} into Eqs. \eqref{deforO1} and \eqref{deforO2} and comparing the coefficients of $t^s~(\forall s\geqslant0)$, we have for all $u,v,w \in V$,
 \begin{eqnarray}
 \label{sys1}&&\sum_{i+j=s,
 \atop i,j\geqslant0}\Big([\frkT_iu,\frkT_jv]-\frkT_i\big(\rho(\frkT_ju)v-\rho(\frkT_jv)u\big)\Big)t^s=0,\\
 \label{sys2}&&\sum_{i+j+k=s,
 \atop i,j,k\geqslant0}\Big(\Courant{\frkT_iu,\frkT_jv,\frkT_kw}-\frkT_i\big(D_{\rho,\mu}(\frkT_ju,\frkT_kv)w+\mu(\frkT_jv,\frkT_kw)u-\mu(\frkT_ju,\frkT_kw)v\big)\Big)t^s=0.
 \end{eqnarray}

 \begin{pro}\label{formal}
 If $T_t=\sum_{i=0}^\infty\frkT_it^i$ is a formal deformation of a relative Rota-Baxter operator $T$ on a Lie-Yamaguti algebra $(\g,[\cdot,\cdot],\Courant{\cdot,\cdot,\cdot})$ with respect to $(V;\rho,\mu)$. Then $\delta^T\frkT_1=0$, i.e., $\frkT_1\in \huaC^1_T(V,\g)$ is a $1$-cocycle of the relative Rota-Baxter operator $T$.
 \end{pro}
 \begin{proof}
 When $s=1$, Eqs. \eqref{sys1} and \eqref{sys2} are equivalent to
 \begin{eqnarray*}
 ~ &&[Tu,\frkT_1v]-[\frkT_1u,Tv]\\
 ~ &=&T\big(\rho(\frkT_1u)v-\rho(\frkT_1v)u\big)+\frkT_1\big(\rho(Tu)v-\rho(Tv)u\big),\\
 ~ &&\\
 ~ &&\Courant{\frkT_1u,Tv,Tw}+\Courant{Tu,\frkT_1v,Tw}+\Courant{Tu,Tv,\frkT_1w}\\
 ~ &=&\frkT_1\Big(D_{\rho,\mu}(Tu,Tv)w+\mu(Tv,Tw)u-\mu(Tu,Tw)v\Big)\\
 ~ &&+T\Big(D_{\rho,\mu}(\frkT_1u,Tv)w+\mu(\frkT_1v,Tw)u-\mu(\frkT_1u,Tw)v\Big)\\
 ~ &&+T\Big(D_{\rho,\mu}(Tu,\frkT_1v)w+\mu(Tv,\frkT_1w)u-\mu(Tu,\frkT_1w)v\Big), \quad \forall u,v,w \in V,
 \end{eqnarray*}
 which implies that $\delta^T(\frkT_1)=0$, i.e., $\frkT_1$ is a $1$-cocycle of $\delta^T.$
 \end{proof}

 \begin{defi}
 Let $T$ be a relative Rota-Baxter operator on a Lie-Yamaguti algebra $(\g,[\cdot,\cdot],\Courant{\cdot,\cdot,\cdot})$ with respect to a representation $(V;\rho,\mu)$. The $1$-cocycle $\frkT_1$ is called the {\bf infinitesimal} of the formal deformation $T_t=\sum_{i=0}^\infty \frkT_it^i$ of $T.$
 \end{defi}

 In the sequel, let us give the notion of equivalent formal deformations of relative Rota-Baxter operators on Lie-Yamaguti algebras.

 \begin{defi}
 Let $T$ be a relative Rota-Baxter operator on a Lie-Yamaguti algebra $(\g,[\cdot,\cdot],\Courant{\cdot,\cdot,\cdot})$ with respect to a representation $(V;\rho,\mu)$. Two formal deformations $\bar T_t=\sum_{i=0}^\infty \bar{\frkT_i}t^i$ and $T_t=\sum_{i=0}^\infty \frkT_it^i$, where $\bar{\frkT_0}=\frkT_0=T$ are said to be {\bf equivalent} if there exist $\frkX\in \wedge^2\g,~\phi_i\in \gl(\g)$ and $\varphi_i\in \gl(V),~i\geqslant2,$ such that for
 \begin{eqnarray}
 \label{equivalent}\phi_t={\Id}_\g+t\frkL_\frkX+\sum_{i=2}^\infty \phi_it^i,~\quad \varphi_t={\Id}_V+tD(\frkX)+\sum_{i=2}^\infty \varphi_it^i,
 \end{eqnarray}
 the following hold:
 \begin{eqnarray}
 [\phi_t(x),\phi_t(y)]=\phi_t[x,y], \quad \Courant{\phi_t(x),\phi_t(y),\phi_t(z)}=\phi_t\Courant{x,y,z}, \quad \forall x,y,z \in \g,
 \end{eqnarray}
 \begin{eqnarray}
 \varphi_t\rho(x)v=\rho(\phi_t(x))(\varphi_t(v)), \quad \varphi_t\mu(x,y)v=\mu(\phi_t(x),\phi_t(y))(\varphi_t(v)), \quad \forall x,y \in \g,~ v \in V,
 \end{eqnarray}
 and
 \begin{eqnarray}
 \label{eq3}T_t\circ \varphi_t=\phi_t\circ \bar{T_t}
 \end{eqnarray}
 as $\mathbb K[[t]]$-module maps.
 \end{defi}

The following theorem is the second key conclusion in this section.
 \begin{thm}\label{thm2}
 Let $T$ be a relative Rota-Baxter operator on a Lie-Yamaguti algebra $(\g,[\cdot,\cdot],\Courant{\cdot,\cdot,\cdot})$ with respect to a representation $(V;\rho,\mu)$. If two formal deformations $\bar T_t=\sum_{i=0}^\infty \bar{\frkT_i}t^i$ and $T_t=\sum_{i=0}^\infty \frkT_it^i$ are equivalent, then their infinitesimals are in the same cohomology classes.
 \end{thm}
 \begin{proof}
 Let $(\phi_t,\varphi_t)$ be the maps defined by \eqref{equivalent}, which makes two deformations $\bar T_t=\sum_{i=0}^\infty \bar{\frkT_i}t^i$ and $T_t=\sum_{i=0}^\infty \frkT_it^i$ equivalent. By \eqref{eq3}, we have
 $$\bar{\frkT_1}v=\frkT_1v+TD(\frkX)v-\Courant{\frkX,Tv}=\frkT_1v+\delta(\frkX)(v),\quad \forall v\in V,$$
 which implies that $\bar{\frkT}_1$ and $\frkT_1$ are in the same cohomology classes.
 \end{proof}

 \begin{defi}
 A relative Rota-Baxter operator $T$ is {\bf rigid} if all formal deformations of $T$ are trivial.
 \end{defi}

 \begin{pro}
 Let $T$ be a relative Rota-Baxter operator on a Lie-Yamaguti algebra $(\g,[\cdot,\cdot],\Courant{\cdot,\cdot,\cdot})$ with respect to a representation $(V;\rho,\mu)$. If $\huaZ^1(V,\g)=\delta(\Nij(T))$, then $T$ is rigid.
 \end{pro}
 \begin{proof}
 Let $T_t=\sum_{i=0}^\infty \frkT_it^i$ be a formal deformation of $T$, then Proposition \ref{formal} gives $\frkT_1\in \huaZ^1(V,\g)$. By the assumption, $\frkT_1=\delta(\frkX)$ for some $\frkX\in \wedge^2\g$. Then setting $\phi_t={\Id}_\g+t\frkL_\frkX,~ \varphi_t={\Id}_V+tD(\frkX)$, we get a formal deformation
 $$\bar{T}_t:=\phi_t^{-1}\circ T_t\circ \varphi_t.$$
 Thus $\bar{T}_t$ is equivalent to $T_t$.  Moreover, we have
 \begin{eqnarray*}
 \bar{T}_t&=&({\Id}-\frkL_\frkX t+(\frkL_\frkX)^2t^2+\cdots+(-1)^{i}(\frkL_\frkX)^{i}t^i+\cdots)(T_t(v+tD(\frkX)v))\\
 ~ &=&Tv+(\frkT_1v+T(D(\frkX)v)-\Courant{\frkX,Tv})t+\bar{\frkT}_2vt^2+\cdots\\
 ~ &=&Tv+\bar{\frkT}_2(v)t^2+\cdots.
 \end{eqnarray*}
 Repeating this procedure, we get that $T_t$ is equivalent to $T$.
 \end{proof}

 \subsection{Higher order deformations of relative Rota-Baxter operators on Lie-Yamaguti algebras}
 In this subsection, we will introduce a special cohomology class associated to an order $n$ deformation of a relative Rota-Baxter operator, and show that a deformation of order $n$ of a relative Rota-Baxter operators is extendable if and only if this cohomology class in the second cohomology group is trivial. Thus we call this cohomology class the obstruction class of a deformation of order $n$ being extendable.

 \begin{defi}
 Let $T$ be a relative Rota-Baxter operator on a Lie-Yamaguti algebra $(\g,[\cdot,\cdot],\Courant{\cdot,\cdot,\cdot})$ with respect to a representation $(V;\rho,\mu)$. If $T_t=\sum_{i=0}^n\frkT_it^i $ with  $\frkT_0=T$, $\frkT_i\in \Hom_{\mathbb K}(V,\g)$, $i=1,2,\cdots,n$, defines a $\mathbb K[t]/(t^{n+1})$-module from $V[t]/(t^{n+1})$ to the Lie-Yamaguti algebra $\g[t]/(t^{n+1})$ satisfying
 \begin{eqnarray}
 \label{ordern1}[T_tu,T_tv]&=&T_t(\rho(T_t)u-\rho(T_tv)u),\\
\label{ordern2}\Courant{T_tu,T_tv,T_tw}&=&T_t(D_{\rho,\mu}(T_tu,T_tv)w+\mu(T_tv,T_tw)u-\mu(T_tu,T_tw)v), \quad \forall u,v,w \in V,
 \end{eqnarray}
 we say that $T_t$ is an {\bf order $n$ deformation} of the relative Rota-Baxter operator $T$.
 \end{defi}

 \begin{rmk}
 The left hand side of Eqs. \eqref{ordern1} and \eqref{ordern2} hold in the Lie-Yamaguti algebra $\g[t]/(t^{n+1})$ and the right hand side of  Eqs. \eqref{ordern1} and \eqref{ordern2} make sense since $T_t$ is a $\mathbb K[t]/(t^{n+1})$-module map.
 \end{rmk}

 \begin{defi}
 Let $T_t=\sum_{i=0}^n\frkT_it^i $ be an order $n$ deformation of a relative Rota-Baxter operator $T$ on a Lie-Yamaguti algebra $(\g,[\cdot,\cdot],\Courant{\cdot,\cdot,\cdot})$ with respect to a representation $(V;\rho,\mu)$. If there exists a $1$-cochain $\frkT_{n+1}\in \Hom_{\mathbb K}(V,\g)$ such that $\widetilde{T_t}=T_t+\frkT_{n+1}t^{n+1}$ is an order $n+1$ deformation of $T$, then we say that $T_t$ is {\bf extendable}.
 \end{defi}

The following theorem is the third key conclusion in this section.
 \begin{thm}\label{ob}
Let $T$ be a relative Rota-Baxter operator on a Lie-Yamaguti algebra $(\g,[\cdot,\cdot],\Courant{\cdot,\cdot,\cdot})$ with respect to a representation $(V;\rho,\mu)$, and $T_t=\sum_{i=0}^n\frkT_it^i$ be an order $n$ deformation of $T$. Then $T_t$ is extendable if and only if the cohomology class ~$[\Ob^T]\in \huaH_T^2(V,\g)$ is trivial, where $\Ob^T=(\Ob_{\rm I}^T,\Ob_{\rm II}^T)\in \huaC_T^2(V,\g)$ is defined by
\begin{eqnarray}
\Ob_{\rm I}^T(v_1,v_2)&=&\sum_{i+j=n+1,\atop i,j\geqslant 1}\Big([\frkT_iv_1,\frkT_jv_2]-\frkT_i(\rho(\frkT_jv_1)v_2-\rho(\frkT_jv_2)v_1)\Big),\\
\Ob_{\rm II}^T(v_1,v_2,v_3)&=&\sum_{i+j+k=n+1,\atop n\geqslant i,j,k\geqslant 0}\Big(\Courant{\frkT_iv_1,\frkT_jv_2,\frkT_kv_3}-\frkT_i(D(\frkT_jv_1,\frkT_kv_2)v_3\\
~ \nonumber&&+\mu(\frkT_jv_2,\frkT_kv_3)v_1-\mu(\frkT_jv_1,\frkT_kv_3)v_2\Big), \quad \forall v_1,v_2,v_3 \in V.
\end{eqnarray}
\end{thm}
\begin{proof}
Let $\widetilde{T_t}=\sum_{i=0}^{n+1}\frkT_it^i$ be the extension of $T_t$, then for all $u,v,w \in V$,
\begin{eqnarray}
\label{n order1}[\widetilde{T_t}u,\widetilde{T_t}v]&=&\widetilde{T_t}\Big(\rho(\widetilde{T_t}u)v-\rho(\widetilde{T_t}v)u\Big),\\
\label{n order2}\Courant{\widetilde{T_t}u,\widetilde{T_t}v,\widetilde{T_t}w}&=&\widetilde{T_t}\Big(D(\widetilde{T_t}u,\widetilde{T_t}v)w+\mu(\widetilde{T_t}v,\widetilde{T_t}w)u
-\mu(\widetilde{T_t}u,\widetilde{T_t}w)v\Big).
\end{eqnarray}
Expanding the Eq.\eqref{n order1} and comparing the coefficients of $t^n$ yields that
\begin{eqnarray*}
\sum_{i+j=n+1,\atop i,j \geqslant 0}\Big([\frkT_iu,\frkT_jv]-\frkT_i\big(\rho(\frkT_ju)v-\rho(\frkT_jv)u\big)\Big)=0,
\end{eqnarray*}
which is equivalent to
\begin{eqnarray*}
~ &&\sum_{i+j=n+1,\atop i,j \geqslant 1}\Big([\frkT_iu,\frkT_jv]-\frkT_i\big(\rho(\frkT_ju)v-\rho(\frkT_jv)u\big)\Big)+[\frkT_{n+1}u,Tv]+[Tu,\frkT_{n+1}v]\\
~ &&\quad-\big(T(\rho(\frkT_{n+1}u)v-\rho(\frkT_{n+1}v)u)+\frkT_{n+1}(\rho(Tu)v-\rho(Tv)u)\big)=0,
\end{eqnarray*}
i.e.,
\begin{eqnarray}
\Ob_{\rm I}^T+\delta_{\rm I}^T(\frkT_{n+1})=0.\label{ob:cocy1}
\end{eqnarray}
Similarly, expanding the Eq.\eqref{n order2} and comparing the coefficients of $t^n$ yields that
\begin{eqnarray}
\Ob_{\rm II}^T+\delta_{\rm II}^T(\frkT_{n+1})=0.\label{ob:cocy2}
\end{eqnarray}
From \eqref{ob:cocy1} and \eqref{ob:cocy2}, we get
$$\Ob^T=-\delta^T(\frkT_{n+1}).$$
Thus, the cohomology class $[\Ob^T]$ is trivial.

Conversely, suppose that the cohomology class $[\Ob^T]$ is trivial, then there exists $\frkT_{n+1}\in \huaC_T^1(V,\g)$, such that ~$\Ob^T=-\delta^T(\frkT_{n+1}).$ Set $\widetilde{T_t}=T_t+\frkT_{n+1}t^{n+1}$. Then for all $0 \leqslant s\leqslant n+1$, ~$\widetilde{T_t}$ satisfies
\begin{eqnarray*}
\sum_{i+j=s}\Big([\frkT_iu,\frkT_jv]-\frkT_i\big(\rho(\frkT_ju)v-\rho(\frkT_jv)u\big)\Big)=0,\\
\sum_{i+j+k=s}\Big(\Courant{\frkT_iu,\frkT_jv,\frkT_kw}-\frkT_i\big(D(\frkT_ju,\frkT_kv)w+\mu(\frkT_jv,\frkT_kw)u-\mu(\frkT_ju,\frkT_kw)v\big)\Big)=0.
\end{eqnarray*}
which implies that $\widetilde{T_t}$ is an order $n+1$ deformation of $T$. Hence it is a extension of $T_t$.
\end{proof}

\begin{defi}
Let $T$ be a relative Rota-Baxter operator on a Lie-Yamaguti algebra $(\g,[\cdot,\cdot],\Courant{\cdot,\cdot,\cdot})$ with respect to a representation $(V;\rho,\mu)$, and $T_t=\sum_{i=0}^n\frkT_it^i$ be an order $n$ deformation of $T$. Then the cohomology class $[\Ob^T]\in \huaH_T^2(V,\g)$ defined in Theorem {\rm \ref{ob}} is called the {\bf obstruction class } of $T_t$ being extendable.
\end{defi}

\begin{cor}
Let $T$ be a relative Rota-Baxter operator on a Lie-Yamaguti algebra $(\g,[\cdot,\cdot],\Courant{\cdot,\cdot,\cdot})$ with respect to a representation $(V;\rho,\mu)$. If $\huaH_T^2(V,\g)=0$, then every $1$-cocycle in $\huaZ_T^1(V,\g)$ is the infinitesimal of some formal deformation of the relative Rota-Baxter operator $T$.
\end{cor}

\end{document}